\makeatletter
\def\input@path{{AOAP Submission/}}
\makeatother
\documentclass[aap]{imsart}

\RequirePackage{amsthm,amsmath,amsfonts,amssymb,mathtools}
\RequirePackage[numbers,sort&compress]{natbib}
\RequirePackage[colorlinks,citecolor=blue,urlcolor=blue]{hyperref}
\RequirePackage{graphicx}
\RequirePackage{tikz}
\usetikzlibrary{arrows.meta,calc,positioning}

\graphicspath{{AOAP Submission/figures/}{figures/}}

\startlocaldefs
\theoremstyle{plain}
\newtheorem{theorem}{Theorem}[section]
\newtheorem{proposition}[theorem]{Proposition}
\newtheorem{corollary}[theorem]{Corollary}
\newtheorem{lemma}[theorem]{Lemma}
\theoremstyle{definition}
\newtheorem{assumption}[theorem]{Assumption}
\newtheorem{definition}[theorem]{Definition}
\newtheorem{remark}[theorem]{Remark}

\newcommand{\R}{\mathbb R}

\newcommand{\E}{\mathbb E}
\newcommand{\Pp}{\mathbb P}
\newcommand{\Law}{\operatorname{Law}}
\newcommand{\TV}{\operatorname{TV}}
\newcommand{\Lip}{\operatorname{Lip}}
\newcommand{\1}{\mathbf{1}}

\newcommand{\eps}{\varepsilon}

\newcommand{\W}{\mathsf W}

\endlocaldefs

\begin{document}
\begin{frontmatter}
\title{Uniform-in-Time Boltzmann Mean-Field Limits for Anchored Binary Opinion Dynamics}
\runtitle{Uniform-in-Time Mean-Field Limits for Opinion Dynamics}
\begin{aug}
\author[A]{\fnms{Jingyi}~\snm{Zhang}\ead[label=e1]{jzhang3450@gatech.edu}}
\author[B]{\fnms{Debankur}~\snm{Mukherjee}\ead[label=e2]{debankur.mukherjee@isye.gatech.edu}}
\address[A]{School of Mathematics, Georgia Institute of Technology\printead[presep={,\ }]{e1}}
\address[B]{H. Milton Stewart School of Industrial and Systems Engineering,
Georgia Institute of Technology\printead[presep={,\ }]{e2}}
\end{aug}

\begin{abstract}
We study a continuous-time opinion model in which agents interact in pairs and each agent has a fixed anchor. At each interaction, both opinions are updated according to a possibly nonlinear rule with random inputs. Under suitable stability and moment assumptions, we establish uniform-in-time propagation of chaos: as the population grows, any fixed number of agents become asymptotically independent, with a common law solving a nonlinear Boltzmann equation. The proof combines a graphical approximation on finite time intervals with exponential convergence of the finite system and its nonlinear limit to their respective stationary laws.

We examine two applications. For an anchored variant of the Friedkin–Johnsen model with random coefficients, we give conditions for sub-Gaussian or power-law tails of stationary deviations from the anchors. Occasional overreaction can produce power-law tails with Gaussian anchors and noise, even as the nonlinear law converges exponentially fast to stationarity. We also study an anchored model of biased assimilation, determining when a nearby opinion moves toward or away from neutrality with the anchor and incoming evidence held at neutral values. Simulations illustrate the tail predictions and show how a single peak in the smoothed empirical opinion profile can split into two.
\end{abstract}

\end{frontmatter}

\section{Introduction}

Opinion dynamics models seek to explain how individual responses to
social influence produce agreement or persistent disagreement
\citep{Friedkin1990,HegselmannKrause2002,Castellano2009}.
Two populations can have the same average opinion even though most
opinions in one are close to that average, while those in the other
form two distinct groups. The frequency of opinions far from an
individual's private baseline is another feature that an average does
not capture. To study these features after many interactions, we need
to connect individual response rules to a description of the population
that remains accurate over long times.

Classical models describe several mechanisms behind these outcomes.
In DeGroot's model, agents update their opinions by taking weighted
averages \citep{DeGroot1974}. The Friedkin--Johnsen model retains an
agent's initial view as a private baseline, or anchor, in subsequent
updates, allowing disagreement to persist
\citep{Friedkin1990,FriedkinJohnsen2011}. Bounded-confidence models
restrict influence to opinions that are sufficiently close, allowing
separate groups to persist
\citep{Deffuant2000,HegselmannKrause2002,Lorenz2007}.
Biased assimilation introduces a different mechanism: agents favor
information that agrees with their current opinions, so the same
evidence can reinforce different views \citep{Dandekar2013}.
The timing of interactions can also affect the evolution;
\citet{ChuPorter2026} study this issue in bounded-confidence models.

\paragraph*{Model}
We consider a population of \(N\) agents who interact in pairs at
random times. Agent \(i\) has opinion \(X_i^N(t)\in\R\) and a fixed
anchor \(B_i\in\R\). The anchor need not coincide with the
initial opinion. Each unordered pair has an independent Poisson
clock, with the same rate for every pair and a fixed total interaction
rate for each agent. At each interaction epoch, both opinions are
updated according to a possibly nonlinear rule that may depend on
both opinions, both anchors, and fresh random inputs. These inputs may
be shared by the two agents or sampled separately for each.

The limiting description is a nonlinear Boltzmann jump process. At
each jump, an agent samples an independent partner from the current
limiting law and applies the same interaction rule. This gives an
equation for the evolving distribution of opinion--anchor pairs. The
main question is whether this description remains accurate after
arbitrarily many interactions. A result on each fixed time interval
does not settle this question: the error bound may deteriorate as the
interval grows, and the finite population may eventually behave
differently from its mean-field limit.

\paragraph*{Main results}
Our main result establishes uniform-in-time propagation of chaos.
Write \(Z_i^N(t)=(X_i^N(t),B_i)\), and let \(f_t\) be the nonlinear
law with initial distribution \(f_0\). When the initial opinion--anchor pairs are
independent with common law \(f_0\), Theorem \ref{thm:main-w1} gives,
under the invariance, contraction, and moment assumptions stated there,
\[
 \lim_{N\to\infty}\sup_{t\ge0}
 W_{1,k}\left(
   \Law(Z_1^N(t),\ldots,Z_k^N(t)),f_t^{\otimes k}
 \right)=0
 \qquad\text{for every fixed }k\ge1.
\]
Here \(W_{1,k}\) is Wasserstein distance of order one on the space of
\(k\) opinion--anchor pairs, as defined in Section
\ref{subsec:state-metrics}. Thus a fixed number of agents become
asymptotically independent, with common law \(f_t\), and the
approximation remains valid at observation times that grow with the
population size.

Our stability assumption requires an interaction to contract the
expected distance between two copies that use the same anchors and
random inputs. A separate drift condition keeps a moment of order greater
than one uniformly bounded. Under these assumptions, the
nonlinear law and the finite-system law converge exponentially fast
to their respective stationary laws, with rates independent of
\(N\). For a fixed interaction rule and anchor distribution, the
stationary nonlinear law is independent of the initial opinions.
Differences in anchors and the fresh random inputs can sustain a spread
of opinions as the influence of the initial opinions fades.

We also prove uniform-in-time convergence in Wasserstein
distance of order two under a contraction condition in mean square.
For this result, we assume either a bounded state space or a drift
condition controlling a moment of order greater than two
(Corollary \ref{cor:w2}).

The proof starts by tracing the past interactions that can affect a
given collection of agents. This graphical argument gives an explicit
estimate in total variation for their joint trajectories on every
finite time interval (Theorem \ref{thm:finite-time-main}). It requires
only measurability of the interaction rule and therefore also applies
to discontinuous rules, such as the confidence threshold in
Deffuant--Weisbuch dynamics \citep{Deffuant2000}. The bound grows exponentially with the
time horizon. To obtain a uniform-in-time bound, we use convergence to
the stationary laws to show that a fixed number of agents are also
asymptotically independent under the finite stationary law. This
controls the error after the initial finite interval.

The strict inequality in the contraction assumption cannot in general
be dropped. In the example in Appendix
\ref{subsec:critical-consensus}, both agents in an interacting pair
adopt the same one of their two opinions, chosen with equal
probability. Opinions remain bounded, and the expected distance
between two copies never increases under the coupling. For an initial
opinion law that is not a point mass, the nonlinear law remains
unchanged and gives the law of any one agent exactly, at every time.
Yet every finite population eventually agrees on a single random
opinion. A spread in the one-agent law can therefore reflect uncertainty
about the eventual consensus, rather than disagreement within a
population.

\paragraph*{Applications}
The uniform approximation also controls empirical averages of bounded
Lipschitz functions. For a large population observed at a late time,
these averages are close to their values under the stationary nonlinear
law for the given anchor distribution. This connects the stationary
law to quantities measured in a single large population. We use two
examples to examine how the response rule affects large departures
from private anchors and the development of two peaks.

The first application concerns an anchored variant of the
Friedkin--Johnsen model in which the coefficient multiplying an
opinion's deviation from its anchor varies randomly from one
interaction to the next.
This coefficient can exceed one, even though the dynamics is
contractive on average. At stationarity, we derive an affine
stochastic recurrence for the deviation of an opinion from its
anchor. We use this recurrence to
give conditions for sub-Gaussian or power-law tails and to characterize
the exponent in the power-law case (Theorem \ref{thm:tail}). The
power-law analysis uses the classical theory of random difference
equations \citep{Kesten1973,Goldie1991,Buraczewski2016}.

The example in Section \ref{sec:numerics} shows the effect of occasional
overreaction while keeping the anchors and noise Gaussian. With a
fixed coefficient less than one, the stationary deviations have
sub-Gaussian tails; allowing occasional overreaction produces
power-law tails. In this model, the response rule can therefore be a
source of extreme opinions even when the private baselines and external
inputs have light tails. Both versions have finite stationary variance
and satisfy our Wasserstein--2 assumptions. Rapid convergence is therefore compatible with very different
frequencies of opinions far from their anchors.

The second application concerns a smooth anchored version of the
biased-assimilation model of \citet{Dandekar2013}. Here reinforcement
of an agent's current view competes with a continuing pull toward a
private anchor. With both the anchor and the incoming evidence held
at neutrality, we give explicit conditions under which a nearby
opinion moves toward or away from the neutral value
(Theorem \ref{thm:modal}).
In simulations of the full binary system, the smoothed empirical
opinion profile develops two peaks even though the initial opinion
distribution and the anchor distribution each have a single peak. These simulations include
parameters outside our contraction assumptions. The theorem concerns
the update with neutral inputs; the observation of two peaks is
numerical.

\paragraph*{Related work}
Our finite-time argument follows the interaction-graph approach of
\citet{GrahamMeleard1997}. For a continuous-time version of Deffuant dynamics,
\citet{GomezSerrano2012} already obtain propagation of chaos in total
variation on path space, including the exact confidence threshold.
Mean-field limits for continuous opinion and gossip models were
studied by \citet{ComoFagnani2011}; see also \citet{Sznitman1991} for
the general theory of propagation of chaos.

For a class of random linear binary updates that contract on average,
\citet{CortezFontbona2016} obtain quantitative Wasserstein estimates
that are uniform in time. Uniform
estimates for Kac's one-dimensional model and for Maxwell molecules
are proved in \citet{Cortez2016Uniform} and
\citet{CortezFontbona2018}. Within
opinion dynamics, \citet{AndreouOlveraCravioto2024} establish a
uniform-in-time approximation, stationary chaos, and commuting
population and time limits for synchronous Friedkin--Johnsen-type
updates on directed random networks with growing average degree. Here we give sufficient conditions for
uniform-in-time chaos and stationary approximation in a class of
nonlinear binary models with fixed anchors and random inputs.

Related uniform-in-time results for diffusions and filtering models
include \citet{DelMoralTugaut2018}, \citet{Durmus2020},
\citet{LackerLeFlem2023}, and \citet{ChenRenWang2022}.
For finite-state mean-field jump processes,
\citet{CohenHuffman2026} obtain uniform-in-time bounds on weak
approximation errors, under suitable regularity assumptions and
exponential stability.

The rest of the paper is organized as follows.
Section \ref{sec:model} introduces the dynamics, notation, and
assumptions. Section \ref{sec:main} states the main results, and
Section \ref{sec:numerics} presents the numerical experiments.
Section \ref{sec:strategy} outlines the proof, and Section
\ref{sec:proofs} gives the details. Appendix \ref{app:classical}
discusses classical models and the role of strict contraction.
Supporting proofs are given in Appendix~\ref{app:technical}.

\section{Model and assumptions}\label{sec:model}

In this section, we describe the finite system and its nonlinear
limit, then introduce the distances and assumptions used in the
main results.

\subsection{Finite population}\label{subsec:finite-system}

Fix \(N\ge2\). Agent \(i\) has an opinion \(X_i^N(t)\in\R\) and an
anchor \(B_i\in\R\). The anchors are sampled at time zero and remain
fixed. We write
\[
 Z_i^N(t)=(X_i^N(t),B_i)\in E:=\R\times\R
\]
for the state of agent \(i\), and set
\(Z^N(t)=(Z_1^N(t),\ldots,Z_N^N(t))\) for the full population.
Each unordered pair \(\{i,j\}\) has an
independent Poisson clock of rate \(2/(N-1)\), and the clocks are
independent of the initial states. Thus each agent interacts at rate
\(2\), and the total interaction rate of the system is \(N\).
Both opinions are updated when their pair's clock rings.

The random inputs at an interaction are collected in a mark with law
\(\theta\) on a measurable space \((U,\mathcal U)\). To use the same
update function for both agents, we specify how the mark changes when
their roles are exchanged. This is described by a measurable map
\(\iota:U\to U\) satisfying
\[
 \iota^2=\mathrm{Id},
 \qquad \theta\circ\iota^{-1}=\theta.
\]
One agent uses the mark \(u\), and the other uses \(\iota u\).
The second identity ensures that these two marks have the same law.

Let
\[
 \Phi:\R^4\times U\to\R
\]
be \(\mathcal B(\R^4)\otimes\mathcal U\)-measurable. In
\(\Phi(x,y;b,c,u)\), the variables \(x,b\) are the opinion and anchor
of the agent being updated, and \(y,c\) are those of its partner.
At a ring time \(t\) of the pair \(\{i,j\}\), with \(i<j\), sample
a mark \(U\sim\theta\), independently of the clocks and the past, and
set
\begin{equation}\label{eq:finite-update}
\begin{aligned}
 X_i^N(t)&=\Phi(X_i^N(t-),X_j^N(t-);B_i,B_j,U),\\
 X_j^N(t)&=\Phi(X_j^N(t-),X_i^N(t-);B_j,B_i,\iota U).
\end{aligned}
\end{equation}
The two updates use the opinions just before the interaction.
Since \(\iota\) preserves \(\theta\), the law of the pair update does
not depend on which agent is listed first.

Taking \(\iota=\mathrm{Id}\) allows both agents to use the same mark.
For independent random inputs, the mark can consist of two independent
components with the same law. In that case, take
\(\iota(u_1,u_2)=(u_2,u_1)\) and let \(\Phi\) use only the first
component of the mark supplied to it.

\subsection{Nonlinear Boltzmann equation}\label{subsec:nonlinear-law}

Write \(\mathcal P(S)\) for the probability laws on a measurable
space \(S\), and \(\langle\varphi,\mu\rangle=\int\varphi\,d\mu\).

To define the limiting update, take two independent opinion--anchor
pairs with a common law \(\mu\in\mathcal P(E)\) and an independent
mark with law \(\theta\). Let \(T\mu\) be the law of the
first agent's opinion and anchor after the update. This is the usual Boltzmann
description of a binary interaction
\citep{GrahamMeleard1997,CortezFontbona2016}. Thus
\begin{equation}\label{eq:Tdef}
 \langle\varphi,T\mu\rangle
 :=\int_{E\times E\times U}
 \varphi(\Phi(x,y;b,c,u),b)\,\theta(du)\,\mu(dx,db)\,\mu(dy,dc),
\end{equation}
for \(\varphi\in B_b(E)\), the space of bounded measurable
real-valued functions on \(E\).

Starting from a law \(f_0\in\mathcal P(E)\), the nonlinear law
\((f_t)_{t\ge0}\) evolves according to
\begin{equation}\label{eq:nl}
 \frac{d}{dt}\langle\varphi,f_t\rangle
 =2\{\langle\varphi,Tf_t\rangle-\langle\varphi,f_t\rangle\},
 \qquad\varphi\in B_b(E).
\end{equation}
The difference \(Tf_t-f_t\) records the change in law due to one
update, and the factor \(2\) is the interaction rate of one agent.
Equivalently,
\begin{equation}\label{eq:mild}
 f_t=e^{-2t}f_0+2\int_0^t e^{-2(t-s)}T(f_s)\,ds,
\end{equation}
where the equality holds after integration against any
\(\varphi\in B_b(E)\). Proposition \ref{prop:wellposed} gives
existence and uniqueness for every initial law \(f_0\).
A stationary law is a fixed point of \(T\). Since the updates leave
the anchors unchanged, \(f_t\) has the same anchor marginal as \(f_0\)
for all \(t\ge0\).

\subsection{State space and distances}\label{subsec:state-metrics}

For the long-time results, fix an anchor law
\(\rho\in\mathcal P(\R)\) and a closed set \(D\subseteq E\).
We will require the dynamics to remain in \(D\), as stated in
Assumption~\ref{ass:inv} below. This includes models on the whole
space, with \(D=E\), and bounded models, such as
\(D=[0,1]\times[0,1]\). Write \(\mathcal P_\rho(D)\) for the
probability laws on \(D\) with anchor marginal \(\rho\), and, for
\(q\ge1\), set
\[
 \mathcal P_{\rho,q}(D):=
 \left\{\mu\in\mathcal P_\rho(D):
 \int_E |x|^q\mu(dx,db)<\infty\right\}.
\]
Uniqueness of a stationary law will always be understood within the
specified domain, anchor distribution, and moment class.

For stability estimates, we compare opinions while keeping their
anchors equal. If \(\mu,\nu\in\mathcal P_\rho(D)\), let
\(\Gamma_\rho(\mu,\nu)\) be the set of laws of
\((X,\widetilde X,B)\) such that \((X,B)\) has law \(\mu\) and
\((\widetilde X,B)\) has law \(\nu\). For \(p\ge1\), define
the anchor-preserving Wasserstein distance by
\begin{equation}\label{eq:condW}
 \W_p^\rho(\mu,\nu)^p
 :=\inf_{\pi\in\Gamma_\rho(\mu,\nu)}
 \int |x-\widetilde x|^p\,\pi(dx,d\widetilde x,db).
\end{equation}
This is a conditional Wasserstein distance, with the anchor as the
conditioning variable \citep{Chemseddine2025}.
It is finite for
\(\mu,\nu\in\mathcal P_{\rho,p}(D)\); no moment condition on
\(\rho\) is needed for this definition.

The propagation-of-chaos statements use the ordinary Wasserstein
distance on the full opinion--anchor space. For \(k\ge1\), let
\(W_{p,k}\) be this distance on \(E^k\), with underlying metric
\begin{equation}\label{eq:Wpk}
 d_{p,k}(z,\widetilde z)^p
 =\sum_{i=1}^k
 \left(|x_i-\widetilde x_i|^p+|b_i-\widetilde b_i|^p\right),
 \qquad z_i=(x_i,b_i),\quad
 \widetilde z_i=(\widetilde x_i,\widetilde b_i).
\end{equation}
We use \(p=1\) in the main theorem and \(p=2\) in
Corollary \ref{cor:w2}. We write \(\Lip(\varphi)\) for the
Lipschitz constant with respect to the metric in use.

For probability measures \(\mu,\nu\) on the same measurable space,
we normalize total variation by
\begin{equation}\label{eq:tv-convention}
 \|\mu-\nu\|_{\TV}
 :=\sup_{\|h\|_\infty\le1}
 \left|\int h\,d\mu-\int h\,d\nu\right|,
\end{equation}
where the supremum is over measurable \(h\). With this convention,
\(\|\mu-\nu\|_{\TV}\in[0,2]\).

\subsection{Assumptions for uniform-in-time convergence}
\label{subsec:uniform-assumptions}

Moment bounds and contraction estimates are also used to obtain
uniform-in-time results for random linear binary interactions
\citep{CortezFontbona2016}. Here we state sufficient conditions for
nonlinear rules with fixed anchors. The first ensures that an
interaction keeps both agents in the chosen domain.

\begin{assumption}[Invariant domain]\label{ass:inv}
The set \(D\subseteq E\) is closed, and for all
\((x,b),(y,c)\in D\),
\[
 (\Phi(x,y;b,c,u),b)\in D
 \qquad\text{for }\theta\text{-a.e. }u.
\]
\end{assumption}

\begin{assumption}[Moment drift]\label{ass:qdrift}
For some \(q>1\), \(\int |b|^q\rho(db)<\infty\), and there exist
constants \(a_x,a_y,c_b,c_c,c_0\ge0\) such that
\begin{equation}\label{eq:qdrift}
 \int_U |\Phi(x,y;b,c,u)|^q\theta(du)
 \le a_x|x|^q+a_y|y|^q+c_b|b|^q+c_c|c|^q+c_0
\end{equation}
for all \((x,b),(y,c)\in D\), with \(a_x+a_y<1\).
\end{assumption}

Starting with \(f_0\in\mathcal P_{\rho,q}(D)\), this condition keeps
the \(q\)-th opinion moment bounded uniformly over time. The same
bound holds for each agent in the finite system started from
\(f_0^{\otimes N}\), independently of \(N\)
(Lemma \ref{lem:moment}).

Finally, we compare two copies of an interaction with the same
anchors and the same random mark.

\begin{assumption}[Anchor-preserving \(L^1\) contraction]
\label{ass:w1contr}
There exist \(\ell_x,\ell_y\ge0\) such that
\begin{equation}\label{eq:w1contr}
 \int_U
 |\Phi(x,y;b,c,u)-\Phi(x',y';b,c,u)|\theta(du)
 \le \ell_x|x-x'|+\ell_y|y-y'|
\end{equation}
whenever \((x,b),(x',b),(y,c),(y',c)\in D\), and
\(\ell:=\ell_x+\ell_y<1\).
\end{assumption}

Applying this bound to both agents shows that the expected sum of
the two absolute opinion differences after the interaction is at most
\(\ell\) times their sum before it. The contraction is thus an
average over the random inputs. For a particular realization of these
inputs, the sum may increase.

Appendix \ref{app:classical} shows how these conditions apply to
anchored linear and smooth bounded-confidence rules. In the convex
linear example, the contraction condition reduces to assigning
positive weight to the anchor on average.

\section{Main results}\label{sec:main}

We first state the uniform-in-time approximation and the finite-time
estimate used in its proof. We then give the Wasserstein--2 result
and turn to the two applications.

\subsection{Uniform-in-time propagation of chaos}\label{subsec:main-uit}

The main theorem below gives an approximation that remains valid however long
the population has been evolving. It also identifies the limiting
stationary law and relates it to empirical averages in a large finite
population.

\begin{theorem}[Uniform-in-time propagation of chaos]\label{thm:main-w1}
Assume Assumptions \ref{ass:inv}, \ref{ass:qdrift}, and
\ref{ass:w1contr}. Let \(f_0\in\mathcal P_{\rho,q}(D)\), and let
\((f_t)_{t\ge0}\) solve \eqref{eq:nl} with initial law \(f_0\).
Start the finite system from \(f_0^{\otimes N}\).

\begin{longlist}[(iii)]
\item[(i)] There is a unique stationary law
\(F_\infty\in\mathcal P_{\rho,q}(D)\), satisfying
\(T(F_\infty)=F_\infty\), and
\begin{equation}\label{eq:w1-relax}
 \W_1^\rho(f_t,F_\infty)
 \le e^{-2(1-\ell)t}\W_1^\rho(f_0,F_\infty),
 \qquad t\ge0.
\end{equation}

\item[(ii)] For every fixed \(k\ge1\),
\begin{equation}\label{eq:uit-poc}
 \lim_{N\to\infty}\sup_{t\ge0}
 W_{1,k}\left(\Law(Z_1^N(t),\ldots,Z_k^N(t)),f_t^{\otimes k}\right)=0.
\end{equation}

\item[(iii)] Let \(\mu_t^N=N^{-1}\sum_{i=1}^N\delta_{Z_i^N(t)}\).
For every bounded Lipschitz function \(\varphi:E\to\R\),
\begin{equation}\label{eq:emp-ft}
 \lim_{N\to\infty}\sup_{t\ge0}
 \E\left|\langle\varphi,\mu_t^N\rangle-\langle\varphi,f_t\rangle\right|=0,
\end{equation}
and
\begin{equation}\label{eq:emp-long}
 \lim_{T\to\infty}\limsup_{N\to\infty}\sup_{t\ge T}
 \E\left|\langle\varphi,\mu_t^N\rangle
       -\langle\varphi,F_\infty\rangle\right|=0.
\end{equation}
\end{longlist}
\end{theorem}

The uniformity in \eqref{eq:uit-poc} concerns the joint law at a common
observation time. In particular, that time may depend on \(N\) and
tend to infinity. The case \(k=2\) gives the empirical conclusion
by controlling the second moment of the centered average;
this is the usual connection between propagation of chaos and empirical
measures \citep{Sznitman1991}. Write \(m_\infty\) for the
opinion marginal of \(F_\infty\): it is the law of \(X\) when
\((X,B)\sim F_\infty\). Taking \(\varphi(x,b)=\psi(x)\) in
\eqref{eq:emp-long} gives the long-time limit of averages of bounded
Lipschitz functions of opinions. The proof is given in Section
\ref{subsec:proof-main}.

The strict inequality \(\ell<1\) cannot in general be replaced
by \(\ell\le1\). Appendix \ref{subsec:critical-consensus} gives
a rule with \(\ell=1\) for which the finite population reaches
random consensus while the nonlinear law remains equal to its
nondegenerate initial law.

\begin{remark}[Empirical observables]\label{rem:empirical-observables}
The moment bound with \(q>1\) allows a truncation argument, so the
conclusions extend to empirical means. For empirical variances, the
truncation argument uses uniform integrability of squared opinions;
a uniform moment bound of some order greater than two is sufficient.
In the long-time
limit, empirical tail probabilities can be recovered at thresholds
where \(m_\infty\) has no atom. A sufficient condition for convergence
in probability of a fixed quantile is that the distribution function
of \(m_\infty\) strictly cross the corresponding probability level;
see \citet[Chapter~21]{vanDerVaart1998}.
The tail exponent and the number of modes require further analysis,
as in the applications below.
\end{remark}

\subsection{Finite-horizon propagation of chaos}\label{subsec:finite-horizon-main}

On a fixed time interval, a graphical construction gives a more general
approximation. It applies to measurable interaction rules, including
discontinuous ones, and compares entire trajectories. This follows the
approach of \citet{GrahamMeleard1997}; for bounded-confidence dynamics,
see also \citet{GomezSerrano2012}.

For \(f_0\in\mathcal P(E)\) and \(0\le T<\infty\), let
\(Q_{f_0,T}\) be the law of the nonlinear process on
\(D([0,T],E)\), the space of c\`adl\`ag \(E\)-valued paths.
This process starts from \(f_0\) and jumps at rate
\(2\). At a jump time \(t\), it samples a partner
\((Y,C)\sim f_t\) and a mark \(U\sim\theta\), independently of one
another and of its past, and replaces \((x,b)\) by
\((\Phi(x,Y;b,C,U),b)\).

\begin{theorem}[Finite-time approximation under measurability]
\label{thm:finite-time-main}
Assume that \(\Phi\) is
\(\mathcal B(\R^4)\otimes\mathcal U\)-measurable, and start the
finite system from \(f_0^{\otimes N}\), where
\(f_0\in\mathcal P(E)\). Then, for every \(0\le T<\infty\),
\(N\ge2\), \(1\le k\le N\), and distinct labels
\(i_1,\ldots,i_k\),
\begin{equation}\label{eq:finite-horizon-path-main}
 \left\|\Law\bigl((Z_{i_1}^N(s))_{0\le s\le T},\ldots,
                  (Z_{i_k}^N(s))_{0\le s\le T}\bigr)
              -Q_{f_0,T}^{\otimes k}\right\|_{\TV}
 \le \frac{k(k+1)(e^{4T}-1)}{N-1}.
\end{equation}
Consequently,
\begin{equation}\label{eq:finite-horizon-marginal-main}
 \sup_{0\le t\le T}
 \left\|\Law(Z_{i_1}^N(t),\ldots,Z_{i_k}^N(t))
              -f_t^{\otimes k}\right\|_{\TV}
 \le \frac{k(k+1)(e^{4T}-1)}{N-1}.
\end{equation}
\end{theorem}

The proof is given in Section \ref{subsec:proof-graphical}.
The error bound grows with \(T\), so this estimate alone gives no control
over arbitrarily long times. In the proof of Theorem
\ref{thm:main-w1}, the moment bound converts it to a Wasserstein--1
estimate on finite intervals. Convergence of the finite and nonlinear
laws to their stationary laws then supplies the control needed at later
times.

\subsection{Convergence in Wasserstein distance of order two}
\label{subsec:w2-main}

A contraction estimate in mean square gives uniform-in-time
approximation in Wasserstein distance of order two. This is a separate
assumption on the interaction rule. It compares the same opinions and
anchors as Assumption \ref{ass:w1contr}, but controls the expected
squared difference after an update.

\begin{assumption}[Contraction in mean square]\label{ass:square}
There exist \(L_x,L_y\ge0\) such that
\begin{equation}\label{eq:square}
 \int_U |\Phi(x,y;b,c,u)-\Phi(x',y';b,c,u)|^2\theta(du)
 \le L_x^2|x-x'|^2+L_y^2|y-y'|^2
\end{equation}
whenever \((x,b),(x',b),(y,c),(y',c)\in D\), with
\(\kappa^2:=L_x^2+L_y^2<1\).
\end{assumption}

\begin{corollary}[Uniform-in-time Wasserstein--2 approximation]
\label{cor:w2}
Assume Assumptions \ref{ass:inv} and \ref{ass:square}, and start the
finite system from \(f_0^{\otimes N}\). If \(D\) is bounded, assume
\(f_0\in\mathcal P_\rho(D)\). If \(D\) is unbounded, assume
Assumption \ref{ass:qdrift} with \(q=r\) for some \(r>2\), and
take \(f_0\in\mathcal P_{\rho,r}(D)\). Let \((f_t)_{t\ge0}\)
solve \eqref{eq:nl} with initial law \(f_0\). Then, for every fixed
\(k\ge1\),
\[
 \lim_{N\to\infty}\sup_{t\ge0}
 W_{2,k}\left(\Law(Z_1^N(t),\ldots,Z_k^N(t)),f_t^{\otimes k}\right)=0.
\]
\end{corollary}

The proof, given in Section \ref{subsec:proof-w2}, uses the moment
bound to convert the finite-time total-variation estimate to a
Wasserstein estimate. Because the cost is now quadratic, a moment of order strictly
greater than two is used on an unbounded domain.

\subsection{Applications to opinion distributions}
\label{subsec:stationary-mechanisms}

The first application concerns the tails of the stationary opinion
distribution relative to the anchors. The second examines how an
interaction affects an opinion near neutrality under biased assimilation.

\paragraph*{Stationary tails}
Consider the following variant of an anchored Friedkin--Johnsen rule:
\begin{equation}\label{eq:rsfj}
 \Phi(x,y;b,c,U)=b+A(x-b)+H(y)+\sigma\xi,
\end{equation}
where \(A\ge0\), \(\sigma\ge0\), and \(H\) is bounded and
Lipschitz. At each update, \(A\) multiplies the current deviation from
the anchor; the partner and the noise contribute the remaining two
terms. The coefficient may exceed one, even when the dynamics contracts
on average. In the finite system, take
\(U=(A_1,\xi_1,A_2,\xi_2)\), where the two pairs are independent
copies of \((A,\xi)\), let \(\iota\) exchange the pairs, and let
\(\Phi\) use the first one.

At stationarity, this update leads to a random difference equation.
The following result uses the classical theory of such equations
\citep{Kesten1973,Goldie1991,Buraczewski2016} to give two sufficient
conditions for different tail behavior.

\begin{theorem}[Stationary tails]\label{thm:tail}
Consider \eqref{eq:rsfj}, with \(H\) bounded and Lipschitz.
Suppose that, for some \(q_0>1\),
\[
 \E|B|^{q_0}<\infty,
 \qquad \E|\xi|^{q_0}<\infty,
 \qquad \E A^{q_0}<1,
 \qquad \E A+\Lip(H)<1.
\]
Then Theorem \ref{thm:main-w1} applies on \(D=E\) with \(q=q_0\)
to every initial law \(f_0\in\mathcal P_{\rho,q_0}(E)\),
and there is a
unique stationary law
\(F_\infty\in\mathcal P_{\rho,q_0}(E)\).
Under \(F_\infty\), the deviation \(R=X-B\) is independent of
\(B\) and satisfies
\begin{equation}\label{eq:R-sre}
 R\stackrel{d}{=}AR'+Q,
 \qquad Q=H(Y)+\sigma\xi.
\end{equation}
Here \(R'\) has the same law as \(R\),
\(Y\sim m_\infty\), and \(R'\), \(Y\), and
\((A,\xi)\) are mutually independent. Dependence between \(A\) and
\(\xi\) is allowed.

\begin{longlist}[(iii)]
\item[(i)] If \(A\le a<1\) almost surely and \(\xi\) has sub-Gaussian
tails, then \(R\) has sub-Gaussian tails. If \(B\) also has
sub-Gaussian tails, the same holds for \(X=B+R\).

\item[(ii)] Suppose that, for some \(\alpha>q_0\),
\[
 \E\log A<0,\qquad \Pp(A>1)>0,\qquad \E A^\alpha=1,
 \qquad 0<\E[A^\alpha\log A]<\infty.
\]
Use the conventions \(\log0=-\infty\) and
\(0^\alpha\log0=0\). Assume that
\(\E|\xi|^{\alpha+\eta}<\infty\) for some \(\eta>0\), that
the conditional law of \(\log A\) given \(A>0\) is non-arithmetic,
and that
\[
 \Pp(Ar+Q=r)<1\qquad\text{for every }r\in\R.
\]
Then there are constants \(C_+,C_-\ge0\), with \(C_++C_->0\),
such that
\[
 \lim_{t\to\infty}t^\alpha\Pp(R>t)=C_+,
 \qquad
 \lim_{t\to\infty}t^\alpha\Pp(R<-t)=C_-.
\]
In particular,
\[
 \Pp(|R|>t)\sim(C_++C_-)t^{-\alpha},
 \qquad t\to\infty.
\]
If also \(\E A^2<1\), then \(\alpha>2\) and \(R\) has finite
variance. The opinion \(X=B+R\) has finite variance whenever
\(B\) does.
\end{longlist}
\end{theorem}

The proof is given in Section \ref{subsec:proof-tail}.
In part~(ii), for every \(p>0\), the positive two-sided tail constant
implies that \(\E|R|^p<\infty\) exactly when \(p<\alpha\). The nonlinear law can therefore converge exponentially fast to
stationarity even when the stationary deviations have power-law tails. The example in Section
\ref{sec:numerics} has \(\alpha>2\) and satisfies the assumptions of
Corollary \ref{cor:w2}. For that example, the continuous component of
the distribution of \(A\) gives non-arithmeticity, and the nonzero
Gaussian noise ensures non-degeneracy.

\paragraph*{Stability of a neutral opinion}
We next consider biased assimilation on \([0,1]\): an agent gives
more weight to evidence that agrees with its current opinion
\citep{Dandekar2013}. Let \(L(r)=(1+e^{-r})^{-1}\), take
\(U=(U_1,U_2)\) with independent standard normal components, and
set \(\iota(U_1,U_2)=(U_2,U_1)\). For
\(\eta\in[0,1]\) and \(\sigma\ge0\), put
\[
 e(y,U)=(1-\eta)L(\sigma U_1)+\eta y.
\]
This combines an external input with the partner's opinion.
For \(\gamma>1\) and \(s>0\), define
\begin{equation}\label{eq:Agamma}
 A_\gamma(x,e)
 =\frac{s x+x^\gamma e}
        {s+x^\gamma e+(1-x)^\gamma(1-e)},
 \qquad 0\le x,e\le1.
\end{equation}
The map \(A_\gamma\) is the biased-assimilation update in~\citet[Eq.~(4)]{Dandekar2013}, written in our notation.
We add a pull toward the anchor:
for \(\beta\in(0,1]\), set
\begin{equation}\label{eq:ba}
 \Phi_\gamma(x,y;b,c,U)
 =(1-\beta)A_\gamma(x,e(y,U))+\beta b,
 \qquad 0\le b,c,x,y\le1.
\end{equation}
The parameter \(\beta\) controls the weight of the anchor. The map
\(A_\gamma\) takes values in \([0,1]\) and is continuously
differentiable on the closed square. Hence the update keeps opinions
in \([0,1]\).

To examine neutrality, hold both the anchor and the incoming evidence
at \(1/2\). The resulting scalar update shows whether one interaction
moves a nearby opinion closer to or farther from \(1/2\).

\begin{theorem}[Stability of the neutral opinion]\label{thm:modal}
Fix \(s>0\), \(\gamma>1\), and \(\beta\in(0,1]\), and define
\[
 G_\gamma(x)=(1-\beta)A_\gamma(x,1/2)+\frac{\beta}{2}.
\]
Then \(G_\gamma(1/2)=1/2\), and
\begin{equation}\label{eq:Lambdagamma}
 G_\gamma'(1/2)
 =(1-\beta)\frac{s+\gamma2^{-\gamma}}{s+2^{-\gamma}}
 =:\Lambda(\gamma).
\end{equation}
If \(\Lambda(\gamma)<1\), the neutral fixed point is locally
attracting. If \(\Lambda(\gamma)>1\), it is locally repelling, and
\(G_\gamma\) has at least two further fixed points
\[
 0<x_-(\gamma)<\frac12<x_+(\gamma)<1,
 \qquad x_-(\gamma)=1-x_+(\gamma).
\]
\end{theorem}

The proof is given in Section \ref{subsec:proof-modal}.
This theorem concerns the update with neutral inputs. It does not
determine the number of modes of the stationary opinion law. For each
fixed anchor law \(\rho\) on \([0,1]\), either contraction assumption
\ref{ass:w1contr} or \ref{ass:square} gives a unique stationary
nonlinear law on \([0,1]^2\). The simulations in Section
\ref{sec:numerics} show two peaks in the smoothed empirical opinion
profile and include parameters outside these contraction assumptions.

\section{Numerics and phase diagrams}\label{sec:numerics}

We now use the two rules from Section \ref{subsec:stationary-mechanisms}
to examine how repeated interactions affect the shape of the opinion
distribution. We first follow the development of two peaks under biased
assimilation, then compare tails with and without occasional
overreaction. Parameter diagrams complement the individual simulations,
as in the study of smooth bounded-confidence models by
\citet{BrooksChodrowPorter2023}.

All simulations use the embedded jump chain of the finite system.
At each step, we choose two distinct agents uniformly and update both
opinions according to \eqref{eq:finite-update}, using their values
before the interaction. The two agents receive independent random
inputs, sampled afresh at each step. Anchors are sampled independently
at the start and remain fixed, and we initialize every opinion at its
anchor. After \(m\) pair updates, we report the exposure
\[
 c=\frac{2m}{N},
\]
the average number of updates per agent. The code accompanying the
figures specifies all random seeds.

\subsection{Development of two peaks}

For the biased-assimilation rule \eqref{eq:ba}, we take
\[
 s=0.05,\qquad \sigma=0.3,\qquad \eta=0.05,
 \qquad B\sim {\rm Beta}(5,5).
\]
Thus the anchor and initial opinion distributions each have a single
peak at \(1/2\). In Figure \ref{fig:modality}, we fix \(\beta=0.13\)
and vary the bias parameter \(\gamma\). We also include \(\gamma=1\),
for which the same update formula is well defined.

We display opinions in the centered coordinate \(m_x=2x-1\).
For each profile, we form a density histogram with \(401\) equal
bins on \([-1,1]\) and smooth it with a discrete Gaussian filter
of standard deviation \(h=0.035\), extending the histogram by zero
outside this interval. A detected peak must have prominence at least
\(0.08\) times the maximum smoothed density. We require a minimum
separation of \(0.07\), rounded up to a whole number of bins, and
discard peaks outside \((-0.96,0.96)\). These conventions are used
for every profile and peak count below; the count refers to the
smoothed sample at the stated exposure.

The left panel of Figure \ref{fig:modality} follows \(N=80{,}000\)
agents with \(\gamma=1.2\) at exposures \(c=0,8,25,55,100\).
The initial central peak gradually gives way to two side peaks.
The right panel compares profiles at \(c=100\), using
\(N=40{,}000\) and
\(\gamma=1.00,1.08,1.16,1.24,1.36\). The first two profiles have
one detected peak; the remaining three have two, on opposite sides
of neutrality.

\begin{figure}[tbp]
\centering
\includegraphics[width=0.95\linewidth]{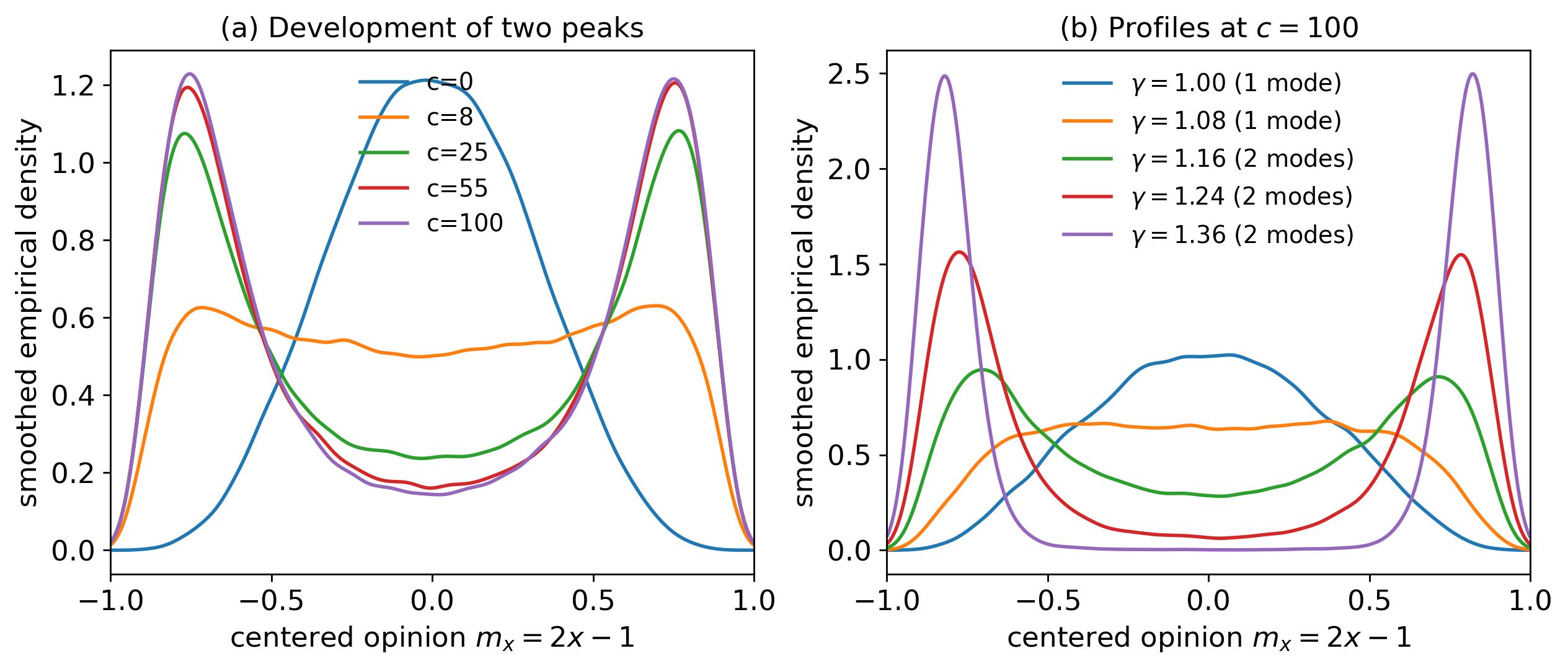}
\caption{Development of two peaks under biased assimilation, with
\(\beta=0.13\). Left: smoothed empirical densities for
\(\gamma=1.2\) and \(N=80{,}000\), at the exposures shown in the
legend. Right: profiles at \(c=100\) for \(N=40{,}000\), as
\(\gamma\) varies. Peak counts use the convention described in the
text. In both panels, \(X_0=B\sim {\rm Beta}(5,5)\).}
\label{fig:modality}
\end{figure}

For comparison, Theorem \ref{thm:modal} gives the curve
\begin{equation}\label{eq:beta-critical-numerics}
 \beta_c(\gamma)
 =1-\frac{s+2^{-\gamma}}{s+\gamma2^{-\gamma}},
\end{equation}
at which \(\Lambda(\gamma)=1\) in \eqref{eq:Lambdagamma}.
When both the anchor and incoming evidence are held at \(1/2\),
the neutral fixed point is locally repelling for
\(\beta<\beta_c(\gamma)\) and locally attracting for
\(\beta>\beta_c(\gamma)\).

Figure \ref{fig:modality-phase} compares this curve with the empirical
peak counts on a grid of \(12\) equally spaced values of
\(\gamma\in[1,1.55]\) and \(10\) equally spaced values of
\(\beta\in[0.05,0.30]\). At each grid point, we run five
independent simulations with \(N=6{,}000\) to exposure \(c=70\),
sampling new anchors for each run. Color records the fraction of runs
with exactly two detected peaks.

The region with two peaks extends above the neutral-stability curve.
The profile at \(\gamma=1.16\), \(\beta=0.13\) in Figure
\ref{fig:modality} gives a concrete example: it has two detected
peaks, although \mbox{\(\Lambda(1.16)\approx0.9952<1\)}.
Thus instability of the update with neutral inputs is not necessary
for the two-peak profiles observed here. The curve concerns one
opinion with its inputs held fixed, whereas the simulated population
also has varying anchors, random evidence, and interactions between
agents.\par

\begin{figure}[tbp]
\centering
\includegraphics[width=0.64\linewidth]{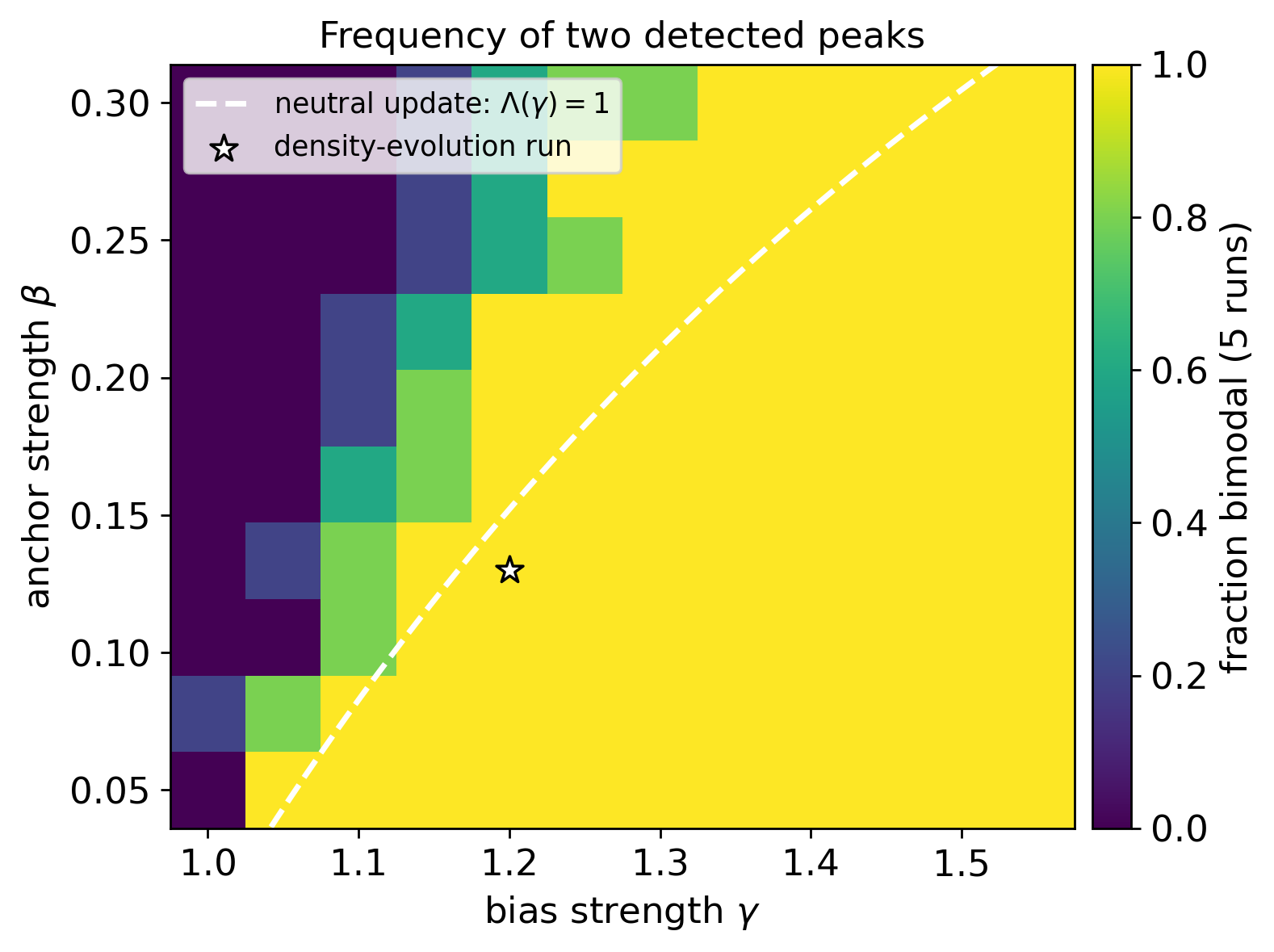}
\caption{Empirical peak counts as bias and anchoring vary. Color gives
the fraction of five independent runs with exactly two detected peaks
at \(c=70\), for \(N=6{,}000\). The dashed curve is
\(\beta=\beta_c(\gamma)\) from
\eqref{eq:beta-critical-numerics}; the update with neutral inputs is
locally attracting above it and repelling below it. The star marks
\((\gamma,\beta)=(1.2,0.13)\), used in the left panel of
Figure \ref{fig:modality}.}
\label{fig:modality-phase}
\end{figure}

\subsection{Stationary tails}

We next consider the anchored rule \eqref{eq:rsfj}, with
\[
 H(y)=0.05\tanh(y),\qquad \sigma=0.1,
 \qquad B\sim N(0,0.2^2),\qquad \xi\sim N(0,1),
\]
and
\[
 A=\begin{cases}
 0.2, & \text{with probability }0.9,\\
 {\rm Unif}[2.5,3.0], & \text{with probability }0.1.
 \end{cases}
\]
Here \(A\) and \(\xi\) are independent. To isolate the effect
of the occasional values \(A>1\), we compare this model with
\(A\equiv0.2\), keeping the other parameters unchanged.

The example with overreaction satisfies both contraction assumptions.
Indeed,
\[
 \E A=0.455,\qquad \E A^2=0.794333\ldots,
 \qquad \E A+\Lip(H)=0.505<1.
\]
For contraction in mean square, put
\(u=(\E A^2)^{1/2}\) and \(L=\Lip(H)=0.05\).
Minkowski's inequality and the weighted Cauchy--Schwarz inequality give
\[
 \E\left|\Phi(x,y;b,c,U)-\Phi(x',y';b,c,U)\right|^2
 \le (u+L)\left(u|x-x'|^2+L|y-y'|^2\right).
\]
Thus Assumption \ref{ass:square} holds with
\(\kappa^2=(u+L)^2=0.885959\ldots<1\).
Also, \(\E A^{2.2}=0.955290\ldots<1\); boundedness of \(H\)
and the Gaussian moments of the anchors and noise then give
Assumption \ref{ass:qdrift} with \(q=2.2\).
Consequently, Corollary \ref{cor:w2} applies to this example.

Theorem \ref{thm:tail}(ii) gives a power-law tail for the stationary
deviation \(R=X-B\), with exponent
\[
 \alpha=2.248206\ldots,
 \qquad \E A^\alpha=1.
\]
The stationary deviation therefore has finite variance.
Theorem \ref{thm:tail}(i) gives
sub-Gaussian tails when \(A\equiv0.2\). This comparison shows that
even within the assumptions of the Wasserstein--2 result, the
stationary distribution can have quite different tail behavior.

For each model, we simulate \(N=90{,}000\) agents to exposure
\(c=260\). The left panel of Figure \ref{fig:tail} plots the
empirical survival function of \(|X-B|\) from one terminal sample,
without smoothing. With rare overreaction, the curve is roughly
parallel to the reference line of slope \(-\alpha\) over much of
the displayed tail range. With \(A\equiv0.2\), it falls much more
rapidly. The reference slope comes from the moment equation above;
it is not fitted to the sample.

\begin{figure}[tbp]
\centering
\includegraphics[width=0.95\linewidth]{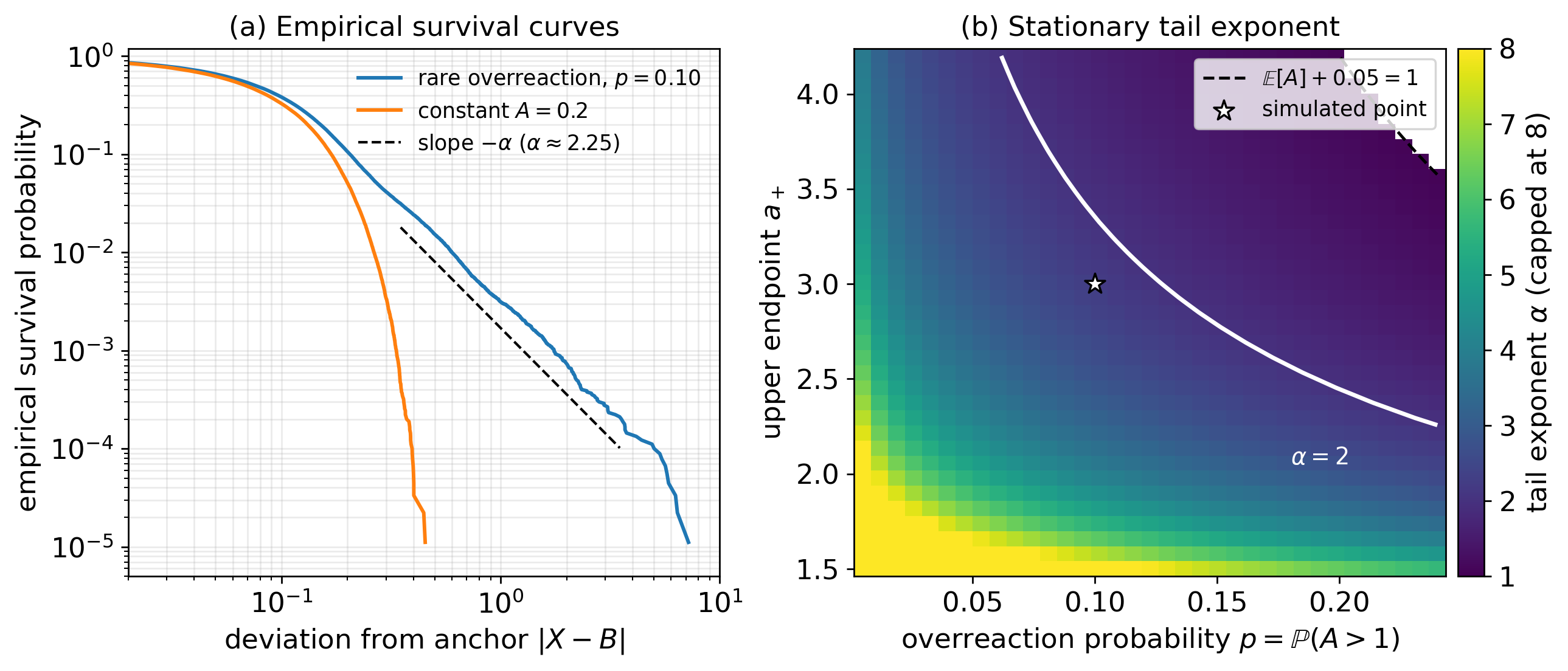}
\caption{The effect of occasional overreaction on the tails. Left:
empirical survival curves of \(|X-B|\) at \(c=260\), for
\(N=90{,}000\). The dashed reference line has slope
\(-\alpha\approx-2.2482\), as predicted by
Theorem \ref{thm:tail}(ii). Right: the exponent computed from
\(\E A^\alpha=1\), capped at \(8\). The white contour is
\(\alpha=2\), and the dashed black curve is
\(\E A+0.05=1\). The star marks
\((p,a_+)=(0.1,3.0)\), used in the simulation with overreaction.}
\label{fig:tail}
\end{figure}

For the right panel, we keep the other parameters as above and compute
the theoretical exponent for the family
\[
 A=\begin{cases}
 0.2, & \text{with probability }1-p,\\
 {\rm Unif}[a_+-0.5,a_+], & \text{with probability }p.
 \end{cases}
\]
We use \(35\) equally spaced values in each of
\(p\in[0.005,0.24]\) and \(a_+\in[1.5,4.2]\). Throughout this
range, \(p=\Pp(A>1)\). Color gives the positive solution of
\(\E A^\alpha=1\), capped at \(8\), on cells satisfying
\(\E A+0.05<1\); this condition also implies
\(\E\log A<0\). Cells that fail this sufficient contraction
condition are left white.
Every colored cell corresponds to a power-law stationary tail.
The white contour \(\alpha=2\) separates finite variance
\((\alpha>2)\) from infinite variance \((\alpha\le2)\).

\section{Proof strategy}\label{sec:strategy}

We outline the proof of the uniform-in-time approximation. The argument
starts on finite intervals, then uses convergence to stationarity to
control later times. An essential step is to compare the stationary
laws of the finite system and its nonlinear limit.

\paragraph*{Approximation on finite intervals}
Fix \(k\) agents and a time horizon \(T\). We trace backward the
interactions that can affect their trajectories on \([0,T]\).
Each partner encountered in this exploration brings its own earlier
interactions. We compare the resulting graph with independent trees
that construct \(k\) nonlinear trajectories. The two constructions
can be coupled to agree until a proposed partner label has already
appeared in the exploration. A branching-process bound controls the
first two moments of the number of labels independently of \(N\).
For fixed \(k\) and \(T\), the probability of a repeated label is
therefore bounded by a constant times \(1/N\). This gives the total-variation estimate
in Theorem \ref{thm:finite-time-main}, using only measurability of the
interaction rule.

\paragraph*{Moment bounds and contraction}
Assumption \ref{ass:qdrift} keeps the opinion moments of order
\(q>1\) bounded uniformly in time and in \(N\). Together with the
assumed anchor moment, this converts the total-variation estimate into
a Wasserstein--1 estimate for the joint laws at a common time
(Lemma \ref{lem:interp}). Assumption \ref{ass:w1contr} then gives
contraction of the nonlinear evolution in \(\W_1^\rho\).
For the finite system, we couple two populations with the same anchors,
pair clocks, and random inputs. The expected average absolute
difference between corresponding opinions contracts at a rate
independent of \(N\). These estimates
give exponential convergence of both systems to their respective
stationary laws.

\paragraph*{Comparing stationary laws}
Let \(F_\infty^N\) be the finite stationary law with anchor marginal
\(\rho^{\otimes N}\). Using the coupling just described, we compare
a system with this law to a second finite system started from
\(F_\infty^{\otimes N}\). Contraction and exchangeability bring
the two \(k\)-agent marginals close. Meanwhile, the graphical
estimate compares the joint law of
\(k\) agents in the second system with \(F_\infty^{\otimes k}\),
since its nonlinear law remains stationary. We choose the observation
time
\[
 s_N=a\log N,\qquad 0<a<\frac14.
\]
Then \(s_N\to\infty\) and \(e^{4s_N}/N\to0\), so both errors
vanish. This proves
that a fixed number of agents become asymptotically independent under
\(F_\infty^N\), with common limiting law \(F_\infty\)
(Lemma \ref{lem:stationary-chaos}).

\paragraph*{Uniformity in time}
We now fix \(T\), independently of \(N\). The finite-time estimate
controls the error on \([0,T]\). For \(t\ge T\), we compare the
finite and nonlinear laws through their stationary laws: the two
errors from convergence to stationarity are small when \(T\) is
large, and the difference between the stationary marginals tends to
zero as \(N\to\infty\). Taking \(N\to\infty\) first and then
\(T\to\infty\) proves Theorem \ref{thm:main-w1}(ii). The empirical
conclusions follow from a second-moment calculation using the case
\(k=2\). For Corollary
\ref{cor:w2}, the same argument uses contraction in mean square;
on an unbounded domain, the conversion from total variation uses
the assumed moment of order greater than two.

\section{Proofs}\label{sec:proofs}

We now give the proofs of the results stated in Section \ref{sec:main},
beginning with the finite-time approximation. Some supporting proofs
are given in Appendix \ref{app:technical}.

\subsection{Finite-time propagation of chaos}\label{subsec:proof-graphical}

We begin with the well-posedness result needed to prove
Theorem \ref{thm:finite-time-main}.

\begin{proposition}[Well-posedness of the nonlinear equation]
\label{prop:wellposed}
Assume that \(\Phi\) is
\(\mathcal B(\R^4)\otimes\mathcal U\)-measurable.
For every \(f_0\in\mathcal P(E)\), equation \eqref{eq:nl} has a
unique solution \((f_t)_{t\ge0}\) continuous in total variation.
The nonlinear jump process described in Section
\ref{subsec:finite-horizon-main}, started from \(f_0\), is unique in
law and has time-\(t\) marginal \(f_t\). Its anchor marginal is
constant in time.
\end{proposition}

The proof is given in Appendix \ref{subsec:technical-wellposed}.
Fix \(N\ge2\), \(1\le k\le N\), distinct labels
\(i_1,\ldots,i_k\), and \(T>0\). Write
\(Q=Q_{f_0,T}\) for the nonlinear path law and set
\begin{equation}\label{eq:epsNk}
 \eps_{N,k}(T):=\frac{k(k+1)(e^{4T}-1)}{N-1}.
\end{equation}

\paragraph*{The paths needed for reconstruction}
To recover the trajectory of an agent, we need the opinion of each
partner just before their interaction. Finding that opinion requires
tracing the partner's earlier interactions as well. We collect this
information by reading the pair clocks backward from \(T\).

\begin{definition}[Backward cone]\label{def:backward-cone}
For a root label \(i\), let \(\mathcal V_0^i=\{i\}\) and use
backward time \(s=T-t\). At a ring at physical time \(T-s\),
ignore the event if neither endpoint is in \(\mathcal V_{s-}^i\).
Otherwise, record the time, pair, and oriented marks, and add any
endpoint not already in \(\mathcal V_{s-}^i\). Each added label
remains active for the rest of the backward exploration.
Let \(\mathcal G_i\) be the resulting marked graph, with the
vertical segment of each label running from physical time zero to
the physical time of its first encounter; the root segment runs to
\(T\).
\end{definition}

A partner first encountered at time \(t\) is needed only up to its
state at \(t-\). Its path after that interaction is not part of
the reconstruction through this encounter. If both endpoints are
already active when we encounter a ring, both of their paths must
continue through it.

\begin{lemma}[Backward-cone reconstruction]\label{lem:cone-reconstruction}
The path \(Z_i^N\) on \([0,T]\) is a measurable function of
\(\mathcal G_i\) and the initial states
\(\{Z_a^N(0):a\in\mathcal V_T^i\}\).
For several roots, their joint paths are determined by the union
of their marked graphs and the initial states on the union of their
label sets.
\end{lemma}

\begin{proof}
Read the recorded events in increasing physical time. At each event,
the two states just before the interaction have already been
reconstructed. Apply \eqref{eq:finite-update} to each endpoint whose
recorded segment continues above that event. When a partner's segment
ends there, only its state just before the interaction is used.
Induction recovers all required segments, including the full root
paths. There are finitely many recorded events, and each update is
measurable, so the reconstruction is measurable.
\end{proof}

\paragraph*{Independent comparison trees}
The comparison process uses a new abstract vertex for every partner.
Unlike a physical label, an abstract vertex is never reused.

\begin{definition}[Boltzmann forest]\label{def:boltzmann-forest}
Start with \(k\) independent roots at physical time \(T\).
In backward time, each active vertex gives birth at rate \(2\),
independently of the others, and remains active after each birth.
A child born at physical time \(t\) carries the partner's path on
\([0,t)\). Each birth receives an independent mark with law
\(\theta\), to be used by the parent. At time zero, assign
independent states with law \(f_0\) to every vertical lineage.
Reconstruct forward using \(\Phi\): at each birth time, update
the parent with the child's state just before that time.
Denote the root paths by \(\overline Z_1,\ldots,\overline Z_k\).
\end{definition}

The number of active vertices is a Yule process with birth rate
\(2n\) in state \(n\). It is nonexplosive, so this reconstruction
uses only finitely many events on \([0,T]\).

\begin{lemma}[Law of the Boltzmann forest]\label{lem:forest-law}
The joint law of the root paths is \(Q^{\otimes k}\).
\end{lemma}

\begin{proof}
The trees are independent, so consider one root and let \(g_t\)
be its law at time \(t\). Discarding events after \(t\) gives
the same law for the root history as a tree constructed with
horizon \(t\).
At a root interaction at time \(s\), the attached subtree is
independent of the root's earlier history and supplies a partner
with law \(g_s\). Conditioning on the last root interaction before
\(t\) gives
\[
 g_t=e^{-2t}f_0+2\int_0^t e^{-2(t-s)}T(g_s)\,ds.
\]
The root interacts at rate \(2\), so \(g_t\) is continuous in
total variation. Proposition \ref{prop:wellposed} therefore gives
\(g_t=f_t\).

This also identifies the path law. Given the root's past, its next
interaction occurs at rate \(2\); at an interaction time \(s\),
the attached subtree and mark supply independent inputs with laws
\(f_s\) and \(\theta\). Thus the root has the transition
mechanism of the nonlinear process, whose law is unique by
Proposition \ref{prop:wellposed}.
\end{proof}

\paragraph*{Assigning physical labels}
Give the roots their prescribed labels \(i_1,\ldots,i_k\).
At a birth from a vertex labelled \(a\), propose a label uniformly
from \(\{1,\ldots,N\}\setminus\{a\}\), using fresh randomness
at each birth. Assign it to the child
if it has not been used. Let \(\tau\) be the first backward time
at which a proposed label has already appeared anywhere in the
forest. At \(\tau\), we stop matching the two constructions.
Figure \ref{fig:forest-coupling} illustrates a repeated proposal.

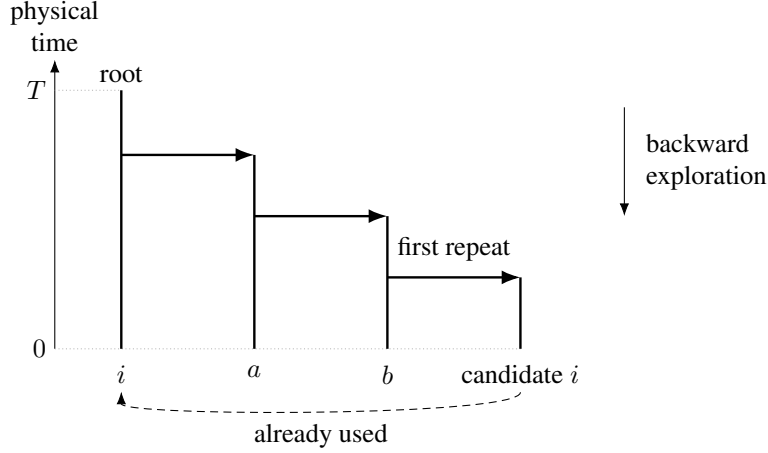
\begin{figure}[tbp]
\centering
\begin{tikzpicture}[
  x=1.1cm,y=0.9cm,>=Latex,
  every node/.style={font=\small},
  history/.style={line width=0.9pt},
  birth/.style={line width=0.9pt,->}
]
  \draw[->] (-0.8,0) -- (-0.8,4.25)
    node[above,align=center] {physical\\time};
  \node[left] at (-0.8,0) {\(0\)};
  \node[left] at (-0.8,3.8) {\(T\)};
  \draw[gray!55,densely dotted] (-0.8,3.8) -- (0,3.8);
  \draw[gray!55,densely dotted] (-0.8,0) -- (4.8,0);

  \draw[history] (0,0) -- (0,3.8);
  \draw[history] (1.6,0) -- (1.6,2.85);
  \draw[history] (3.2,0) -- (3.2,1.95);
  \draw[history] (4.8,0) -- (4.8,1.05);
  \draw[birth] (0,2.85) -- (1.6,2.85);
  \draw[birth] (1.6,1.95) -- (3.2,1.95);
  \draw[birth] (3.2,1.05) -- (4.8,1.05)
    node[midway,above=3pt] {first repeat};

  \node[above] at (0,3.8) {root};
  \node[anchor=north] (rootLabel) at (0,-0.12) {\(i\)};
  \node[anchor=north] at (1.6,-0.12) {\(a\)};
  \node[anchor=north] at (3.2,-0.12) {\(b\)};
  \node[anchor=north] (candidateLabel) at (4.8,-0.12)
    {candidate \(i\)};

  \draw[densely dashed,->]
    (candidateLabel.south)
    .. controls (4.8,-1.0) and (0,-1.0) ..
    node[midway,below=2pt] {already used}
    (rootLabel.south);

  \draw[->] (6.05,3.55) -- (6.05,1.95)
    node[midway,right=5pt,align=left] {backward\\exploration};
\end{tikzpicture}
\caption{Backward exploration with proposed physical labels. Each
birth creates a fresh abstract vertex, whose path is needed only
before its attachment time. The last candidate label is \(i\),
which is already in use, so the matching stops. The dashed arrow
records this repeated label; it does not represent a finite pair-clock
ring.}
\label{fig:forest-coupling}
\end{figure}

\begin{lemma}[Coupling up to a repeated label]\label{lem:forest-coupling}
The finite backward graphs rooted at \(i_1,\ldots,i_k\) and the
labelled forest can be coupled to agree before \(\tau\).
If the finite system starts from \(f_0^{\otimes N}\), their initial
states can also be coupled so that
\[
 (Z_{i_1}^N,\ldots,Z_{i_k}^N)
 = (\overline Z_1,\ldots,\overline Z_k)
 \quad\text{on }\{\tau>T\}.
\]
\end{lemma}

\begin{proof}
Suppose that \(y\) distinct physical labels are currently active.
A vertex labelled \(a\) proposes any particular unused label
\(j\) at rate \(2/(N-1)\), exactly the rate of the finite pair
clock \(\{a,j\}\). These events can therefore be matched.

Rings between two active finite labels require a little care.
The total rate of repeated proposals in the forest is
\(2y(y-1)/(N-1)\). The corresponding finite rings have total rate
\[
 \binom y2\frac{2}{N-1}=\frac{y(y-1)}{N-1}.
\]
Retain each repeated proposal with probability \(1/2\), using
independent coins, to produce these finite rings. For each unordered
pair, its two proposal directions then give the required rate
\(2/(N-1)\). Every repeated proposal stops the matching, including
one rejected by this thinning. Thus \(\tau\) need not be the time
of an actual finite ring.

The marks can be matched as well. If the active vertex \(a\) sees
forest mark \(u\) and proposes \(j\), use unordered-pair mark
\(u\) when \(a<j\), and \(\iota u\) when \(j<a\).
In both cases, \(a\) sees \(u\), and
\(\theta\circ\iota^{-1}=\theta\) gives the correct finite mark
law. These rate and mark assignments construct the joint backward
exploration up to \(\tau\). 
Beyond \(\tau\), complete the two backward explorations separately, using fresh Poisson increments at their prescribed rates and fresh independent marks with law \(\theta\). This gives the required conditional laws given the history already revealed.

Before a repeated proposal, the assigned physical labels are
distinct. Since this assignment is independent of the initial
states, the corresponding finite and abstract lineages can receive
identical \(f_0\) states while preserving their product laws.
Use independent states for any remaining abstract lineages.
On \(\{\tau>T\}\), Lemma \ref{lem:cone-reconstruction} then
gives equality of the root paths.
\end{proof}

\begin{lemma}[Conflict probability]\label{lem:conflict}
The first repeated proposal satisfies
\[
 \Pp(\tau\le T)\le\frac12\eps_{N,k}(T).
\]
\end{lemma}

\begin{proof}
Let \(Y_s\) be the number of active vertices in the unrestricted
\(k\)-root forest at backward time \(s\). Before \(\tau\),
each of the \(Y_s\) vertices gives birth at rate \(2\), and the
proposed label is already in use with probability
\((Y_s-1)/(N-1)\). The counting process stopped at its first
repeated proposal therefore gives
\[
 \Pp(\tau\le T)
 =\frac{2}{N-1}\int_0^T
 \E[\1_{\{s<\tau\}}Y_s(Y_s-1)]\,ds
 \le\frac{2}{N-1}\int_0^T\E[Y_s(Y_s-1)]\,ds.
\]
For this Yule process, let
\(m_1(s)=\E Y_s\) and \(m_2(s)=\E[Y_s(Y_s-1)]\).
The birth rate \(2n\) gives
\[
 m_1'=2m_1,\qquad m_2'=4m_2+4m_1,
 \qquad m_1(0)=k,\quad m_2(0)=k(k-1).
\]
Hence
\[
 m_1(s)=ke^{2s},\qquad
 m_2(s)=k(k+1)e^{4s}-2ke^{2s}.
\]
Substituting and dropping the negative term yields
\[
 \Pp(\tau\le T)
 \le\frac{2k(k+1)}{N-1}\int_0^T e^{4s}\,ds
 =\frac12\eps_{N,k}(T).
\]
\end{proof}

\begin{proof}[Proof of Theorem \ref{thm:finite-time-main}]
For \(T>0\), Lemmas \ref{lem:forest-law} and
\ref{lem:forest-coupling} couple the finite tagged paths to the
forest root paths, whose joint law is \(Q^{\otimes k}\), with
equality on \(\{\tau>T\}\).
The coupling inequality for our total-variation convention gives
\[
 \left\|\Law(Z_{i_1}^N,\ldots,Z_{i_k}^N)-Q^{\otimes k}\right\|_{\TV}
 \le 2\Pp(\tau\le T)\le\eps_{N,k}(T).
\]
Here all paths are restricted to \([0,T]\). This proves
\eqref{eq:finite-horizon-path-main}. Evaluation at any fixed time
is measurable and cannot increase total variation, giving
\eqref{eq:finite-horizon-marginal-main}. For \(T=0\), both claims
follow directly from the product initial law.
\end{proof}

\subsection{Moment bounds and contraction}\label{subsec:proof-contraction}

We now establish the estimates used to control the dynamics over long
times. The moment bound also allows us to pass from the total-variation
estimate of the preceding subsection to Wasserstein distance.

\paragraph*{Moment bounds}
The drift assumption bounds the moment after one interaction by a
fraction of the current moment, plus a contribution from the anchors
and the random inputs. Iterating this bound gives the following result.

\begin{lemma}[Uniform moment bounds]\label{lem:moment}
Under Assumptions \ref{ass:inv} and \ref{ass:qdrift}, let
\(f_0\in\mathcal P_{\rho,q}(D)\), and set
\[
\begin{gathered}
 a:=a_x+a_y,\qquad C:=(c_b+c_c)\int |b|^q\rho(db)+c_0,\\
 K:=\max\left\{\int |x|^qf_0(dx,db),\frac{C}{1-a}\right\}.
\end{gathered}
\]
The nonlinear solution remains in \(\mathcal P_{\rho,q}(D)\), with
\[
 \sup_{t\ge0}\int |x|^qf_t(dx,db)\le K.
\]
For the finite system started from \(f_0^{\otimes N}\),
\[
 \sup_{N\ge2}\sup_{t\ge0}
 \E\frac1N\sum_{i=1}^N|X_i^N(t)|^q\le K.
\]
By exchangeability, \(\E|X_i^N(t)|^q\le K\) for every agent
\(i\) and time \(t\ge0\).
\end{lemma}

The proof is given in Appendix \ref{subsec:technical-moments}.
The next lemma explains how this moment control is used. We can
couple two laws so that the probability of different states is half
their total-variation distance. The moment bound then controls the
expected cost on that event.

\begin{lemma}[From total variation to Wasserstein distance]\label{lem:interp}
Let \(p\ge1\), \(r>p\), and \(\mu,\nu\in\mathcal P(E^k)\).
If
\[
 \int d_{p,k}(z,0)^r\mu(dz)
 +\int d_{p,k}(z,0)^r\nu(dz)\le M,
\]
then
\[
 W_{p,k}(\mu,\nu)
 \le C_{p,r,k}M^{1/r}
       \|\mu-\nu\|_{\TV}^{1/p-1/r}.
\]
Here \(C_{p,r,k}<\infty\) depends only on \(p,r,k\).
If instead both laws are supported on a set of diameter at most
\(R\) under \(d_{p,k}\), then
\[
 W_{p,k}(\mu,\nu)
 \le R\|\mu-\nu\|_{\TV}^{1/p}.
\]
\end{lemma}

The proof is given in Appendix \ref{subsec:technical-interpolation}.
For the main theorem, we use \(p=1\) and \(r=q\). The opinion
bound above and the anchor moment in Assumption \ref{ass:qdrift}
give the required moment on \(E^k\). For the unbounded-domain
case of Corollary \ref{cor:w2}, we use \(p=2\) and the assumed
moment of order \(r>2\).

\paragraph*{Nonlinear contraction}
To compare two nonlinear updates, we couple the opinions while
keeping their anchors equal. We also use the same random mark
and coupled partners. The partners remain independent of the
agents being updated, as required by the definition of \(T\).

\begin{lemma}[Contraction of the nonlinear update]\label{lem:Tcontr}
\leavevmode
\begin{longlist}[(iii)]
\item[(i)] Under Assumption \ref{ass:w1contr}, if
\(\mu,\nu,T\mu,T\nu\in\mathcal P_{\rho,1}(D)\), then
\[
 \W_1^\rho(T\mu,T\nu)\le\ell\,\W_1^\rho(\mu,\nu).
\]
\item[(ii)] Under Assumption \ref{ass:square}, if
\(\mu,\nu,T\mu,T\nu\in\mathcal P_{\rho,2}(D)\), then
\[
 \W_2^\rho(T\mu,T\nu)\le\kappa\,\W_2^\rho(\mu,\nu).
\]
\end{longlist}
\end{lemma}

\begin{proof}
Let \((X,\widetilde X,B)\) be an anchor-preserving coupling of
\(\mu\) and \(\nu\), and let \((Y,\widetilde Y,C)\) be an
independent copy. Take an independent mark \(U\sim\theta\),
and write
\[
 X^+=\Phi(X,Y;B,C,U),\qquad
 \widetilde X^+=\Phi(\widetilde X,\widetilde Y;B,C,U).
\]
Then \((X^+,\widetilde X^+,B)\) couples \(T\mu\) and
\(T\nu\) with equal anchors. Assumption \ref{ass:w1contr} gives
\[
 \E|X^+-\widetilde X^+|
 \le\ell_x\E|X-\widetilde X|
     +\ell_y\E|Y-\widetilde Y|
 =\ell\,\E|X-\widetilde X|.
\]
Taking the infimum over the initial coupling proves \textup{(i)}.
Under Assumption \ref{ass:square}, the same construction gives
\[
 \E|X^+-\widetilde X^+|^2
 \le\kappa^2\E|X-\widetilde X|^2,
\]
which proves \textup{(ii)}.
\end{proof}

In our applications, the required moments of \(T\mu\) and
\(T\nu\) follow from the
drift assumption (with order greater than two for \(W_2\)) or
boundedness of \(D\). The next lemma uses the resulting moment
control to compare two nonlinear evolutions.

\begin{lemma}[Contraction of the nonlinear evolution]
\label{lem:nonlinear-contract}
Let \((f_t)\) and \((g_t)\) be nonlinear solutions supported on
\(D\), with anchor marginal \(\rho\).
\begin{longlist}[(iii)]
\item[(i)] Suppose Assumption \ref{ass:w1contr} holds and the first
opinion moments of both solutions are bounded on each finite time
interval. Then
\[
 \W_1^\rho(f_t,g_t)
 \le e^{-2(1-\ell)t}\W_1^\rho(f_0,g_0).
\]
\item[(ii)] Suppose Assumption \ref{ass:square} holds and the second
opinion moments of both solutions are bounded on each finite time
interval. Then
\[
 \W_2^\rho(f_t,g_t)
 \le e^{-(1-\kappa^2)t}\W_2^\rho(f_0,g_0).
\]
\end{longlist}
\end{lemma}

\begin{proof}
Start from any anchor-preserving coupling of \(f_0\) and \(g_0\).
We evolve the triple \((X,\widetilde X,B)\) at rate \(2\), using
the joint update in the preceding proof. At each interaction, the
partner triple is an independent sample from the current joint law.
The Picard argument in Appendix \ref{subsec:technical-wellposed}
applies to this measurable update on triples. Its two marginals
solve \eqref{eq:nl}, so uniqueness identifies them as \(f_t\)
and \(g_t\).

For \(\delta_1(t):=\E|X_t-\widetilde X_t|\), the mild equation
of the coupled law and the preceding calculation give
\[
 \delta_1(t)\le e^{-2t}\delta_1(0)
 +2\ell\int_0^t e^{-2(t-s)}\delta_1(s)\,ds.
\]
The assumed moment bounds justify integration of this unbounded
cost by Tonelli's theorem. Multiplying by \(e^{2t}\) and applying
Gronwall's lemma yields
\(\delta_1(t)\le e^{-2(1-\ell)t}\delta_1(0)\).
Taking the infimum over the initial coupling proves \textup{(i)}.

For \(\delta_2(t):=\E|X_t-\widetilde X_t|^2\), Assumption
\ref{ass:square} gives instead
\[
 \delta_2(t)\le e^{-2t}\delta_2(0)
 +2\kappa^2\int_0^t e^{-2(t-s)}\delta_2(s)\,ds,
\]
and hence
\(\delta_2(t)\le e^{-2(1-\kappa^2)t}\delta_2(0)\).
Taking the infimum and then the square root proves \textup{(ii)}.
\end{proof}

\paragraph*{Finite-system contraction}
For the finite system, we use the same anchors, pair clocks, and
marks in two copies. The initial opinions may be dependent within
each population. This lets us compare an evolving population with
a stationary one.

\begin{lemma}[Finite-population synchronous contraction]
\label{lem:finite-contract}
Couple two finite systems with initial states in \(D^N\), equal
anchors coordinatewise, and common pair clocks and marks, independent
of their joint initial state. Define
\[
 D_N^{(1)}(t):=\frac1N\sum_{i=1}^N
 |X_i^N(t)-\widetilde X_i^N(t)|,
 \qquad
 D_N^{(2)}(t):=\frac1N\sum_{i=1}^N
 |X_i^N(t)-\widetilde X_i^N(t)|^2.
\]
\begin{longlist}[(iii)]
\item[(i)] Under Assumptions \ref{ass:inv} and \ref{ass:w1contr},
if \(\E D_N^{(1)}(0)<\infty\), then
\[
 \E D_N^{(1)}(t)
 \le e^{-2(1-\ell)t}\E D_N^{(1)}(0).
\]
\item[(ii)] Under Assumptions \ref{ass:inv} and \ref{ass:square},
if \(\E D_N^{(2)}(0)<\infty\), then
\[
 \E D_N^{(2)}(t)
 \le e^{-2(1-\kappa^2)t}\E D_N^{(2)}(0).
\]
\end{longlist}
\end{lemma}

\begin{proof}
Consider the two coupled chains after \(m\) pair updates, and let
\(\Delta_i\) be the difference between their \(i\)-th opinions.
Let \(\mathcal F_m\) be generated by the initial states and these
first \(m\) updates. If the next pair is \(\{i,j\}\), Assumption
\ref{ass:w1contr} and \(\theta\circ\iota^{-1}=\theta\) give
\[
 \E\bigl[|\Delta_i^+|+|\Delta_j^+|
       \mid\mathcal F_m,\ \text{next pair }\{i,j\}\bigr]
 \le\ell\bigl(|\Delta_i|+|\Delta_j|\bigr),
\]
where the superscript \(+\) denotes the differences after the
update. The orientation identity allows the same assumption to be
applied to the endpoint using \(\iota U\).

Let \(\widehat D_{N,m}^{(1)}=N^{-1}\sum_i|\Delta_i|\).
Each agent belongs to the uniformly chosen pair with probability
\(2/N\). Averaging the last estimate over that pair gives
\[
 \E[\widehat D_{N,m+1}^{(1)}\mid\mathcal F_m]
 \le\left(1-\frac{2(1-\ell)}N\right)
       \widehat D_{N,m}^{(1)}.
\]
Iterate this inequality. The total interaction rate is \(N\), so
the number \(J_t\) of updates by time \(t\) is Poisson with mean
\(Nt\), independently of the embedded chain. Averaging over
\(J_t\) yields
\[
 \E D_N^{(1)}(t)
 \le\E\left[\left(1-\frac{2(1-\ell)}N\right)^{J_t}\right]
       \E D_N^{(1)}(0)
 =e^{-2(1-\ell)t}\E D_N^{(1)}(0).
\]

Under Assumption \ref{ass:square}, the pair calculation holds with
squared differences and factor \(\kappa^2\). Thus
\[
 \E[\widehat D_{N,m+1}^{(2)}\mid\mathcal F_m]
 \le\left(1-\frac{2(1-\kappa^2)}N\right)
       \widehat D_{N,m}^{(2)},
\]
where \(\widehat D_{N,m}^{(2)}=N^{-1}\sum_i|\Delta_i|^2\).
Averaging over the same Poisson event count proves \textup{(ii)}.
\end{proof}

\subsection{Stationary laws and stationary chaos}\label{subsec:proof-stationary}

The contraction estimates allow us to construct stationary laws for
both systems. We then compare these laws using a finite population
started from independent samples of the nonlinear stationary law.

Throughout this subsection, assume Assumptions \ref{ass:inv},
\ref{ass:qdrift}, and \ref{ass:w1contr}, and suppose that
\(\mathcal P_{\rho,q}(D)\ne\varnothing\). Write
\[
 M_*:=\frac{(c_b+c_c)\int |b|^q\rho(db)+c_0}{1-a_x-a_y}.
\]
This will bound the stationary opinion moments. The completeness
result used below is given in Appendix
\ref{subsec:technical-completeness}.

\paragraph*{The nonlinear stationary law}
We obtain a fixed point of \(T\) by repeatedly applying the update
to an initial law. Contraction makes the iterates converge, and
the drift bound keeps their moments under control.

\begin{lemma}[Stationary law of the nonlinear equation]\label{lem:mf-relax}
There is a unique stationary law in \(\mathcal P_{\rho,1}(D)\).
This law, denoted by \(F_\infty\), belongs to
\(\mathcal P_{\rho,q}(D)\) and satisfies
\[
 TF_\infty=F_\infty,\qquad
 \int |x|^qF_\infty(dx,db)\le M_*.
\]
For every nonlinear solution started from
\(f_0\in\mathcal P_{\rho,q}(D)\),
\[
 \W_1^\rho(f_t,F_\infty)
 \le e^{-2(1-\ell)t}\W_1^\rho(f_0,F_\infty),
 \qquad t\ge0.
\]
\end{lemma}

\begin{proof}
Choose \(\mu_0\in\mathcal P_{\rho,q}(D)\) and set
\(\mu_{n+1}=T\mu_n\). The iterates remain in
\(\mathcal P_{\rho,q}(D)\). With \(a=a_x+a_y<1\) and
\(M_n=\int |x|^q\mu_n(dx,db)\), the drift bound gives
\[
 M_{n+1}\le aM_n+(1-a)M_*,
 \qquad
 M_n\le a^nM_0+(1-a^n)M_*.
\]
Also, Lemma \ref{lem:Tcontr}\textup{(i)} yields
\[
 \W_1^\rho(\mu_{n+1},\mu_n)
 \le\ell^n\W_1^\rho(\mu_1,\mu_0).
\]
Thus the iterates are Cauchy in \(\mathcal P_{\rho,1}(D)\).
By Lemma \ref{lem:conditional-completeness}, they converge in
\(\W_1^\rho\), and hence weakly, to a law \(F_\infty\) in
this space. Lower semicontinuity gives
\[
 \int |x|^qF_\infty(dx,db)
 \le\liminf_{n\to\infty}M_n\le M_*.
\]
In particular, \(F_\infty\in\mathcal P_{\rho,q}(D)\), so
the contraction of \(T\) applies to \(\mu_n\) and
\(F_\infty\). It follows that
\(T\mu_n\to TF_\infty\) in \(\W_1^\rho\).
Since \(T\mu_n=\mu_{n+1}\to F_\infty\), this proves
\(TF_\infty=F_\infty\).

If \(G\in\mathcal P_{\rho,1}(D)\) is another fixed point,
then
\[
 \W_1^\rho(G,F_\infty)
 =\W_1^\rho(TG,TF_\infty)
 \le\ell\,\W_1^\rho(G,F_\infty),
\]
so \(G=F_\infty\). Finally, Lemmas \ref{lem:moment} and
\ref{lem:nonlinear-contract}\textup{(i)}, with
\(g_t\equiv F_\infty\), give the convergence estimate.
\end{proof}

\paragraph*{The finite stationary law}
For laws on \(D^N\) with anchor marginal \(\rho^{\otimes N}\)
and finite first opinion moment, define
\[
 \mathcal W_{1,N}(\Pi,\widetilde\Pi)
 :=\inf\E\left[\frac1N\sum_{i=1}^N|X_i-\widetilde X_i|\right],
\]
where the infimum is over couplings with the same anchor vector.
Let \((P_t^N)\) be the finite-system semigroup, so that
\(\Pi P_t^N\) is the law at time \(t\) when the initial law is
\(\Pi\).

\begin{lemma}[Stationary law of the finite system]\label{lem:finite-invariant}
For every \(N\ge2\), the finite system has a unique invariant
law \(F_\infty^N\) on \(D^N\) with anchor marginal
\(\rho^{\otimes N}\) and finite first opinion moment. This law
is exchangeable and satisfies
\[
 \int\frac1N\sum_{i=1}^N|x_i|^q\,F_\infty^N(dz)\le M_*.
\]
For every law \(\Pi\) in the same first-moment class,
\[
 \mathcal W_{1,N}(\Pi P_t^N,F_\infty^N)
 \le e^{-2(1-\ell)t}\mathcal W_{1,N}(\Pi,F_\infty^N),
 \qquad t\ge0.
\]
\end{lemma}

\begin{proof}
We first check that \(P_t^N\) preserves this first-moment class.
The reference population started from \(F_\infty^{\otimes N}\)
has its average \(q\)-th opinion moment bounded by \(M_*\),
by Lemma \ref{lem:moment}. Any law \(\Pi\) in the stated class
can be coupled to this product law with identical anchor vectors
and finite expected average absolute difference. Apply common
updates, independently of the joint initial state. Lemma
\ref{lem:finite-contract} bounds this difference at every time,
so the population started from \(\Pi\) also has finite first
opinion moment.

Taking the infimum over initial couplings in the same lemma gives
\[
 \mathcal W_{1,N}(\Pi P_t^N,\widetilde\Pi P_t^N)
 \le e^{-2(1-\ell)t}\mathcal W_{1,N}(\Pi,\widetilde\Pi).
\]
This space is complete by Lemma \ref{lem:conditional-completeness}.
Hence the map \(\Pi\mapsto\Pi P_1^N\) has a unique fixed point,
which we denote by \(F_\infty^N\). For any \(s\ge0\),
the semigroup property gives
\[
 (F_\infty^N P_s^N)P_1^N
 =(F_\infty^N P_1^N)P_s^N=F_\infty^N P_s^N.
\]
Thus \(F_\infty^N P_s^N\) is another fixed point of the same
map, and uniqueness gives \(F_\infty^N P_s^N=F_\infty^N\).
This proves stationarity for all times. The contraction estimate
then gives the stated convergence bound.

The fixed-point iteration started from \(F_\infty^{\otimes N}\)
gives
\(\Pi_n:=F_\infty^{\otimes N}P_n^N\to F_\infty^N\)
in \(\mathcal W_{1,N}\). Its moments satisfy
\[
 \int\frac1N\sum_i|x_i|^q\,\Pi_n(dz)\le M_*,
 \qquad n\ge0.
\]
Convergence in \(\mathcal W_{1,N}\) implies weak convergence.
Lower semicontinuity therefore gives the same bound for
\(F_\infty^N\).

Finally, the dynamics are unchanged in law by relabelling the
agents. If a permutation reverses the order of a pair, the mark
is replaced by \(\iota U\), which has the same law as \(U\).
Every permuted version of \(F_\infty^N\) is consequently an
invariant law in the same class. Uniqueness makes
\(F_\infty^N\) exchangeable.
\end{proof}

\paragraph*{Comparing the stationary laws}
We now start one finite population in stationarity and a second
from \(F_\infty^{\otimes N}\). Contraction brings the two
populations close. The finite-time approximation compares the
second population with its nonlinear limit, whose law remains
\(F_\infty\).

\begin{lemma}[Stationary propagation of chaos]\label{lem:stationary-chaos}
For every fixed \(k\ge1\),
\[
 W_{1,k}\bigl((F_\infty^N)^{(k)},F_\infty^{\otimes k}\bigr)
 \longrightarrow0\qquad\text{as }N\to\infty,
\]
where \((F_\infty^N)^{(k)}\) is the marginal on the first
\(k\) coordinates.
\end{lemma}

\begin{proof}
Couple \(F_\infty^N\) and \(F_\infty^{\otimes N}\) with
identical anchor vectors, using their conditional laws given the
common anchor vector. Average this coupling over simultaneous
permutations of the agents. Both marginals are exchangeable, so
this preserves the marginals and equality of anchors, and makes
the coupling exchangeable.

Run the two populations with common pair clocks and marks,
independently of their joint initial state. Denote their states by
\(Z^N(t)\) and \(\widetilde Z^N(t)\), respectively. The common
updates preserve the exchangeability of the coupling, by the same
relabelling argument used above. The stationary moment bounds give
\[
 \E D_N^{(1)}(0)\le2M_*^{1/q}.
\]
For \(N\ge\max\{2,k\}\), Lemma \ref{lem:finite-contract}
and equality of the anchors therefore yield
\[
\begin{aligned}
 &W_{1,k}\bigl((F_\infty^N)^{(k)},
     \Law(\widetilde Z_1^N(s),\ldots,\widetilde Z_k^N(s))\bigr)\\
 &\qquad\le k\E D_N^{(1)}(s)
 \le2kM_*^{1/q}e^{-2(1-\ell)s}.
\end{aligned}
\]

The second population starts from a product law. Theorem
\ref{thm:finite-time-main} applies, and its nonlinear law is
\(F_\infty\) at every time. Lemmas \ref{lem:moment} and
\ref{lem:interp}, together with the anchor moment, then give
\[
 W_{1,k}\bigl(\Law(\widetilde Z_1^N(s),\ldots,
     \widetilde Z_k^N(s)),F_\infty^{\otimes k}\bigr)
 \le C_k\eps_{N,k}(s)^{1-1/q},
\]
where \(C_k\) is independent of \(s\) and \(N\). Combining
the two estimates gives, for every \(s\ge0\),
\[
 W_{1,k}\bigl((F_\infty^N)^{(k)},F_\infty^{\otimes k}\bigr)
 \le2kM_*^{1/q}e^{-2(1-\ell)s}
     +C_k\eps_{N,k}(s)^{1-1/q}.
\]
Take \(s_N=\tfrac18\log N\). Then \(s_N\to\infty\) and
\[
 \eps_{N,k}(s_N)
 =\frac{k(k+1)(N^{1/2}-1)}{N-1}\longrightarrow0.
\]
Both terms vanish, proving the lemma.
\end{proof}

The particular exponent in the graphical estimate is not essential
here. The same argument works with a finite-time total-variation bound
of the form \(C_ke^{cs}/N\), where \(c>0\) is fixed: take
\(s_N=a\log N\) with \(0<a<1/c\).

\subsection{Proof of the main theorem}\label{subsec:proof-main}

\begin{proof}[Proof of Theorem \ref{thm:main-w1}]
Part~\textup{(i)} is given by Lemma \ref{lem:mf-relax}.

\emph{Uniformity in time.}
Fix \(k\ge1\) and take \(N\ge\max\{2,k\}\).
For each fixed \(T>0\), Theorem \ref{thm:finite-time-main}
and Lemmas \ref{lem:moment} and \ref{lem:interp} give
\[
 \sup_{0\le t\le T}
 W_{1,k}\bigl(\Law(Z_1^N(t),\ldots,Z_k^N(t)),f_t^{\otimes k}\bigr)
 \le C_{k,f_0}\eps_{N,k}(T)^{1-1/q}
 \longrightarrow0.
\]
Here and below, constants are independent of \(N,t,T\).
The moment bound needed for interpolation includes both the opinions,
by Lemma \ref{lem:moment}, and the anchors, whose common law
\(\rho\) has a finite \(q\)-th moment.

For later times, we compare the two systems through their stationary
laws. Couple \(f_0^{\otimes N}\) and \(F_\infty^N\) with the
same anchor vector and average over simultaneous permutations of
the agents, as in the proof of Lemma \ref{lem:stationary-chaos}.
Run the populations with common pair clocks and marks, independently
of their joint initial state. This preserves the exchangeability
of the coupling. The initial discrepancy satisfies
\[
 \E D_N^{(1)}(0)
 \le \left(\int |x|^qf_0(dx,db)\right)^{1/q}+M_*^{1/q},
\]
where \(M_*\) is the stationary moment bound from the preceding
subsection. Lemma \ref{lem:finite-contract} therefore gives
\[
 W_{1,k}\bigl(\Law(Z_1^N(t),\ldots,Z_k^N(t)),
                  (F_\infty^N)^{(k)}\bigr)
 \le k\E D_N^{(1)}(t)
 \le C_{k,f_0}e^{-2(1-\ell)t}.
\]
For the nonlinear law, taking \(k\) independent copies of a coupling
with equal anchors and using part~\textup{(i)} gives
\[
 W_{1,k}(f_t^{\otimes k},F_\infty^{\otimes k})
 \le k\W_1^\rho(f_t,F_\infty)
 \le C_{k,f_0}e^{-2(1-\ell)t}.
\]
The triangle inequality now yields
\[
\begin{aligned}
 &\sup_{t\ge T}
 W_{1,k}\bigl(\Law(Z_1^N(t),\ldots,Z_k^N(t)),f_t^{\otimes k}\bigr)\\
 &\qquad\le C_{k,f_0}e^{-2(1-\ell)T}
       +W_{1,k}\bigl((F_\infty^N)^{(k)},F_\infty^{\otimes k}\bigr).
\end{aligned}
\]
By Lemma \ref{lem:stationary-chaos}, the last term tends to zero
as \(N\to\infty\). Combining this with the bound on \([0,T]\),
we obtain
\[
 \limsup_{N\to\infty}\sup_{t\ge0}
 W_{1,k}\bigl(\Law(Z_1^N(t),\ldots,Z_k^N(t)),f_t^{\otimes k}\bigr)
 \le C_{k,f_0}e^{-2(1-\ell)T}.
\]
Letting \(T\to\infty\) proves part~\textup{(ii)}.

\emph{Empirical averages.}
Let \(\varphi:E\to\R\) be bounded and Lipschitz, and write
\[
 \bar\varphi_t=\langle\varphi,f_t\rangle,
 \qquad
 \mathcal E_N(t)=\langle\varphi,\mu_t^N\rangle-\bar\varphi_t.
\]
We estimate its second moment directly. Set
\[
 \psi_t(z_1,z_2)
   =(\varphi(z_1)-\bar\varphi_t)(\varphi(z_2)-\bar\varphi_t).
\]
Then \(\int\psi_t\,df_t^{\otimes2}=0\), and the Lipschitz
constant of \(\psi_t\) for \(d_{1,2}\) is at most
\(2\|\varphi\|_\infty\Lip(\varphi)\), uniformly in \(t\).
Exchangeability gives, for \(N\ge2\),
\[
\begin{aligned}
 \E|\mathcal E_N(t)|^2
 &=\frac1N\E(\varphi(Z_1^N(t))-\bar\varphi_t)^2
   +\frac{N-1}{N}\E\psi_t(Z_1^N(t),Z_2^N(t))\\
 &\le\frac{4\|\varphi\|_\infty^2}{N}
   +2\|\varphi\|_\infty\Lip(\varphi)
      W_{1,2}\bigl(\Law(Z_1^N(t),Z_2^N(t)),f_t^{\otimes2}\bigr).
\end{aligned}
\]
Part~\textup{(ii)} with \(k=2\) and Cauchy--Schwarz imply
\(\sup_{t\ge0}\E|\mathcal E_N(t)|\to0\), proving
\eqref{eq:emp-ft}. Finally, part~\textup{(i)} yields
\[
\begin{aligned}
 &\sup_{t\ge T}\E\left|
   \langle\varphi,\mu_t^N\rangle-\langle\varphi,F_\infty\rangle
   \right|\\
 &\qquad\le\sup_{t\ge0}\E|\mathcal E_N(t)|
   +\Lip(\varphi)\W_1^\rho(f_0,F_\infty)e^{-2(1-\ell)T}.
\end{aligned}
\]
Taking the upper limit as \(N\to\infty\) and then letting
\(T\to\infty\) proves \eqref{eq:emp-long} and completes the proof.
\end{proof}

\subsection{Proof of the Wasserstein--2 corollary}\label{subsec:proof-w2}

\begin{proof}[Proof of Corollary \ref{cor:w2}]
We follow the preceding proof, using contraction in mean square.
We first construct stationary laws under the assumptions of the
corollary.

\emph{Stationary laws and convergence.}
Start with \(\mu_0=f_0\) and set \(\mu_{n+1}=T\mu_n\).
On a bounded domain, the moments of these iterates are bounded
automatically. On an unbounded domain, \eqref{eq:T-moment-drift}
with \(q=r\) bounds their \(r\)-th moments uniformly in \(n\).
Lemma \ref{lem:Tcontr}\textup{(ii)} gives
\[
 \W_2^\rho(\mu_{n+1},\mu_n)
 \le\kappa^n\W_2^\rho(\mu_1,\mu_0).
\]
The space \(\mathcal P_{\rho,2}(D)\) is complete by Lemma
\ref{lem:conditional-completeness}. Thus the iterates converge
in \(\W_2^\rho\) to a law \(F_\infty\). As in the proof of
Lemma \ref{lem:mf-relax}, lower semicontinuity preserves the
moment bound, and contraction of \(T\) gives
\(TF_\infty=F_\infty\). The fixed point is unique in
\(\mathcal P_{\rho,2}(D)\). The second moments of \(f_t\)
are uniformly bounded; on an unbounded domain this follows from
Lemma \ref{lem:moment} with \(q=r\). Lemma
\ref{lem:nonlinear-contract}\textup{(ii)} now yields
\[
 \W_2^\rho(f_t,F_\infty)
 \le e^{-(1-\kappa^2)t}\W_2^\rho(f_0,F_\infty).
\]

For laws on \(D^N\) with anchor marginal \(\rho^{\otimes N}\)
and finite second opinion moment, use the distance
\[
 \mathcal W_{2,N}(\Pi,\widetilde\Pi)^2
 :=\inf\E\left[\frac1N\sum_{i=1}^N|X_i-\widetilde X_i|^2\right],
\]
where the anchors agree coordinatewise in the coupling.
Coupling to a population started from \(F_\infty^{\otimes N}\)
shows that \(P_t^N\) preserves this moment class: the reference
population has bounded second moments, and Lemma
\ref{lem:finite-contract}\textup{(ii)} controls the squared
difference. Taking the infimum over initial couplings gives
\[
 \mathcal W_{2,N}(\Pi P_t^N,\widetilde\Pi P_t^N)
 \le e^{-(1-\kappa^2)t}\mathcal W_{2,N}(\Pi,\widetilde\Pi).
\]
By completeness, \(\Pi\mapsto\Pi P_1^N\) has a unique fixed
point, denoted by \(F_\infty^N\). The semigroup and relabelling
arguments in the proof of Lemma \ref{lem:finite-invariant} give
stationarity at every time and exchangeability.

On an unbounded domain, apply Lemma \ref{lem:moment} with
\(q=r\) to the population started from \(F_\infty^{\otimes N}\).
Its normalized \(r\)-th opinion moment is bounded independently
of time and \(N\). Passing to the invariant limit by lower
semicontinuity gives the same bound for \(F_\infty^N\).
On a bounded domain, this control is automatic. In particular,
in either case there is a constant \(M_2<\infty\), independent
of \(N\), such that
\[
 \int|x|^2F_\infty(dx,db)\le M_2,
 \qquad
 \int\frac1N\sum_i|x_i|^2\,F_\infty^N(dz)\le M_2.
\]

\emph{Approximation on finite intervals.}
Write
\[
 \delta:=
 \begin{cases}
  1/2, & D\text{ bounded},\\
  1/2-1/r, & D\text{ unbounded}.
 \end{cases}
\]
In both cases, \(\delta>0\). For fixed \(k\) and
\(N\ge\max\{2,k\}\), Theorem \ref{thm:finite-time-main}
and Lemma \ref{lem:interp} give
\[
 \sup_{0\le t\le T}
 W_{2,k}\bigl(\Law(Z_1^N(t),\ldots,Z_k^N(t)),f_t^{\otimes k}\bigr)
 \le C_{k,f_0}\eps_{N,k}(T)^\delta.
\]
The constant is independent of \(N\) and \(T\). On an unbounded
domain, the required \(r\)-th moments of the opinions and anchors
follow from Lemma \ref{lem:moment} and the assumed anchor moment. The same estimate
applies with initial law \(F_\infty^{\otimes N}\), whose
nonlinear law remains \(F_\infty\).

\emph{Stationary chaos and uniformity in time.}
Use the exchangeable coupling of \(F_\infty^N\) and
\(F_\infty^{\otimes N}\) from the proof of Lemma
\ref{lem:stationary-chaos}, with equal anchors and common clocks
and marks independent of the joint initial state. Its initial
squared discrepancy satisfies \(\E D_N^{(2)}(0)\le4M_2\).
Writing \(\widetilde Z^N\) for the population started from
\(F_\infty^{\otimes N}\), Lemma
\ref{lem:finite-contract}\textup{(ii)} gives
\[
\begin{aligned}
 &W_{2,k}\bigl((F_\infty^N)^{(k)},
       \Law(\widetilde Z_1^N(s),\ldots,\widetilde Z_k^N(s))\bigr)\\
 &\qquad\le\bigl(k\E D_N^{(2)}(s)\bigr)^{1/2}
 \le2\sqrt{kM_2}\,e^{-(1-\kappa^2)s}.
\end{aligned}
\]
Combining this with the finite-time estimate yields
\[
 W_{2,k}\bigl((F_\infty^N)^{(k)},F_\infty^{\otimes k}\bigr)
 \le2\sqrt{kM_2}\,e^{-(1-\kappa^2)s}
       +C_k\eps_{N,k}(s)^\delta.
\]
Take \(s_N=\tfrac18\log N\). Then \(s_N\to\infty\) and
\(\eps_{N,k}(s_N)\to0\), so the stationary marginals converge
in \(W_{2,k}\).

Finally, couple the population started from \(f_0^{\otimes N}\)
to a stationary copy in the same way. Their initial expected
squared discrepancy is bounded independently of \(N\), so
their first \(k\) coordinates are within
\(C_{k,f_0}e^{-(1-\kappa^2)t}\) in \(W_{2,k}\).
Also,
\[
 W_{2,k}(f_t^{\otimes k},F_\infty^{\otimes k})
 \le\sqrt{k}\,\W_2^\rho(f_t,F_\infty).
\]
The time split in the preceding proof therefore gives, for every
\(T>0\),
\[
 \limsup_{N\to\infty}\sup_{t\ge0}
 W_{2,k}\bigl(\Law(Z_1^N(t),\ldots,Z_k^N(t)),f_t^{\otimes k}\bigr)
 \le C_{k,f_0}e^{-(1-\kappa^2)T}.
\]
Letting \(T\to\infty\) proves the corollary.
\end{proof}

\subsection{Proofs for the two applications}
\label{subsec:proof-applications}
\label{subsec:proof-tail}
\label{subsec:proof-modal}

\paragraph*{Stationary tails}
\begin{proof}[Proof of Theorem \ref{thm:tail}]
We first check the assumptions of the main theorem on \(D=E\).
The estimate
\[
 |\Phi(x,y;b,c,U)-\Phi(x',y';b,c,U)|
 \le A|x-x'|+\Lip(H)|y-y'|
\]
gives Assumption \ref{ass:w1contr} with
\(\ell_x=\E A\) and \(\ell_y=\Lip(H)\).
For the moment bound, write
\(\Phi=Ax+(1-A)b+H(y)+\sigma\xi\). Young's inequality and
boundedness of \(H\) give, for every \(\eps>0\),
\[
\begin{aligned}
 \E|\Phi(x,y;b,c,U)|^{q_0}
 &\le(1+\eps)\E A^{q_0}|x|^{q_0}\\
 &\quad+C_\eps\bigl(
      \E|1-A|^{q_0}|b|^{q_0}+1+\E|\xi|^{q_0}\bigr).
\end{aligned}
\]
Choose \(\eps\) so that
\((1+\eps)\E A^{q_0}<1\). This gives Assumption
\ref{ass:qdrift} with \(a_y=0\), and Theorem
\ref{thm:main-w1} gives the stationary law \(F_\infty\).

Let \(R=X-B\) under this law, and let \(\nu_b\) be the
conditional law of \(R\) given \(B=b\). One update sends \(R\) to
\[
 R^+=AR+Q,\qquad Q=H(Y)+\sigma\xi,
\]
where \(Y\sim m_\infty\), and \(Y\) and \((A,\xi)\) are
independent of each other and of \((R,B)\). This affine kernel
does not depend on \(b\). It preserves finite first moments and
contracts the usual \(W_1\) distance on \(\R\) by the factor
\(\E A<1\), so it has at most one invariant law with finite first
moment.

Both \(X\) and \(B\) have finite \(q_0\)-th moments.
Disintegrating \(TF_\infty=F_\infty\) therefore gives, for
\(\rho\)-almost every \(b\), an invariant law \(\nu_b\)
with finite first moment for this same affine kernel.
Uniqueness makes these conditional laws agree. Thus \(R\) is
independent of \(B\), and taking an independent copy \(R'\)
before the update gives \eqref{eq:R-sre} with the stated
independence properties.

\emph{Sub-Gaussian tails.}
For part~\textup{(i)}, choose \(a\in(0,1)\) with \(A\le a\)
almost surely. Let \((A_n,Q_n)\) be independent copies of
\((A,Q)\), and consider
\[
 S:=\sum_{n\ge1}\left(\prod_{j=1}^{n-1}A_j\right)Q_n.
\]
The series converges absolutely almost surely and in \(L^1\),
since its absolute terms are bounded by \(a^{n-1}|Q_n|\).
Splitting off the first term gives \(S=Q_1+A_1S'\), where
\(S'\) has the same law as \(S\) and is independent of
\((A_1,Q_1)\). Its law is therefore invariant for the affine
update, and uniqueness gives \(S\stackrel d=R\).

Boundedness of \(H\) and the sub-Gaussian tails of \(\xi\)
give \(\E e^{cQ^2}<\infty\) for some \(c>0\).
Put \(w_n=(1-a)a^{n-1}\), so that \(\sum_nw_n=1\).
Since \((1-a)|S|\le\sum_nw_n|Q_n|\), Jensen's inequality gives
\[
 \E e^{c(1-a)^2R^2}
 =\E e^{c(1-a)^2S^2}
 \le\sum_{n\ge1}w_n\E e^{cQ_n^2}
 =\E e^{cQ^2}<\infty.
\]
Thus \(R\) has sub-Gaussian tails. If \(B\) also has
sub-Gaussian tails, the bound \(|X|\le|B|+|R|\) gives the
same conclusion for \(X\).

\emph{Power-law tails.}
For part~\textup{(ii)}, boundedness of \(H\) and the noise assumption give
\[\E|Q|^{\alpha+\eta}<\infty.\]
The assumptions on \(A\), together with this moment bound,
allow us to apply Goldie's implicit renewal theorem
\citep[Theorem~4.1]{Goldie1991} to \eqref{eq:R-sre}. It gives the two
one-sided limits in the statement, with \(C_+,C_-\ge0\).
The non-degeneracy assumption ensures \(C_++C_->0\).
Adding the two limits gives the asserted tail asymptotic for
\(|R|\).

If \(\E A^2<1\) and \(\alpha\le2\), then
\[
 1=\E A^\alpha\le(\E A^2)^{\alpha/2}<1,
\]
a contradiction. Hence \(\alpha>2\), and the tail asymptotic
gives \(\E R^2<\infty\). If \(B\) has finite variance as well,
then \(X=B+R\) has finite variance.
\end{proof}

\paragraph*{Stability of the neutral opinion}
\begin{proof}[Proof of Theorem \ref{thm:modal}]
The symmetry \(G_\gamma(1-x)=1-G_\gamma(x)\) gives
\(G_\gamma(1/2)=1/2\). Differentiating \eqref{eq:Agamma}
at \(x=e=1/2\) yields
\[
 G_\gamma'(1/2)
 =(1-\beta)\frac{s+\gamma2^{-\gamma}}{s+2^{-\gamma}}
 =\Lambda(\gamma)\ge0.
\]
This proves \eqref{eq:Lambdagamma}. Since the derivative is
nonnegative, the usual derivative criterion gives local attraction
when \(\Lambda(\gamma)<1\) and local repulsion when
\(\Lambda(\gamma)>1\).

Assume now that \(\Lambda(\gamma)>1\), and put
\(H_\gamma(x)=G_\gamma(x)-x\). We have
\(H_\gamma(1/2)=0\) and \(H_\gamma'(1/2)>0\), so
\(H_\gamma(x)>0\) just to the right of \(1/2\).
On the other hand,
\[
 H_\gamma(1)=G_\gamma(1)-1=-\frac\beta2<0.
\]
The intermediate value theorem gives a fixed point
\(x_+(\gamma)\in(1/2,1)\). Reflection gives the second fixed
point \(x_-(\gamma)=1-x_+(\gamma)\in(0,1/2)\).
\end{proof}

\begin{appendix}

\section{Classical models and boundary cases}\label{app:classical}

We first give a general way to check contraction, then apply the
assumptions to several familiar rules. The examples at the boundary
\(\ell=1\) show that the behavior over long times can depend on
more than the contraction coefficient. We suppress arguments of
\(\Phi\) that a particular rule does not use.

\paragraph*{A derivative criterion}
For each anchor \(b\), let \(I_b=\{x:(x,b)\in D\}\), and
suppose every nonempty \(I_b\) is an interval. For each \(b,c\)
with nonempty \(I_b,I_c\), assume that, outside a
\(\theta\)-null set of marks, \(\Phi(\cdot,\cdot;b,c,u)\)
is continuously differentiable on a neighborhood of \(I_b\times I_c\).
Set
\[
\begin{aligned}
 \ell_x&:=\sup_{(x,b),(y,c)\in D}
       \int_U|\partial_x\Phi(x,y;b,c,u)|\theta(du),\\
 \ell_y&:=\sup_{(x,b),(y,c)\in D}
       \int_U|\partial_y\Phi(x,y;b,c,u)|\theta(du).
\end{aligned}
\]
Integrating each partial derivative along its coordinate segment
and then averaging over the mark gives \eqref{eq:w1contr}.
Thus \(\ell_x+\ell_y<1\) is sufficient for Assumption
\ref{ass:w1contr}. The domain and moment assumptions must also
be checked, as in the examples below.

\subsection{Anchored linear rules}\label{subsec:linear-rules}

Consider an anchored linear rule motivated by the
Friedkin--Johnsen model \citep{Friedkin1990,FriedkinJohnsen2011},
with random coefficients:
\[
 \Phi(x,y;b,c,U)
   =(1-\beta_U-\omega_U)x+\omega_Uy+\beta_Ub.
\]
For real coefficients, Assumption \ref{ass:w1contr} holds if
\[
 \E|1-\beta_U-\omega_U|+\E|\omega_U|<1.
\]
To check the remaining assumptions, consider the convex case
\[
 \beta_U,\omega_U\ge0,\qquad
 \beta_U+\omega_U\le1\quad\text{almost surely}.
\]
The coefficient sum above is then \(1-\E\beta_U\), so positive
anchoring on average gives strict contraction. The full space
\(E\) is invariant, as is \(I\times I\) for any closed
interval \(I\) containing the opinions and anchors.

For \(q>1\), convexity gives
\[
 |\Phi(x,y;b,c,U)|^q
 \le(1-\beta_U-\omega_U)|x|^q
     +\omega_U|y|^q+\beta_U|b|^q.
\]
If \(\rho\) has a finite \(q\)-th moment and
\(\E\beta_U>0\), Assumption \ref{ass:qdrift} therefore holds
with
\[
 a_x=\E(1-\beta_U-\omega_U),\qquad
 a_y=\E\omega_U,\qquad c_b=\E\beta_U,
 \qquad a_x+a_y=1-\E\beta_U<1,
\]
and \(c_c=c_0=0\). Theorem \ref{thm:main-w1} then applies
to every \(f_0\in\mathcal P_{\rho,q}(E)\).

The same anchoring condition also gives contraction in mean square.
With \(a_U=1-\beta_U-\omega_U\), weighted Cauchy--Schwarz gives,
for \(r,s\in\R\),
\[
 |a_Ur+\omega_Us|^2
 \le(1-\beta_U)(a_Ur^2+\omega_Us^2).
\]
Thus Assumption \ref{ass:square} holds with
\[
 L_x^2=\E[(1-\beta_U)(1-\beta_U-\omega_U)],\qquad
 L_y^2=\E[(1-\beta_U)\omega_U],
\]
and \(\kappa^2=\E(1-\beta_U)^2<1\) whenever
\(\E\beta_U>0\). Corollary \ref{cor:w2} applies on \(E\)
if the initial opinions and anchors have finite moments of some
order \(r>2\). It also applies on a bounded invariant domain
\(I\times I\), without an additional moment assumption.

\subsection{Smooth bounded-confidence rules}
\label{subsec:smooth-confidence}

Let \(0<\beta\le1\) and consider
\[
 \Phi(x,y;b,c)=(1-\beta)x+\beta b+\psi(y-x),
\]
where \(\psi\) is bounded and \(L\)-Lipschitz. This includes
smooth confidence cutoffs; the estimates below do not require
differentiability. Directly,
\[
 |\Phi(x,y;b,c)-\Phi(x',y';b,c)|
 \le(1-\beta+L)|x-x'|+L|y-y'|.
\]
Hence \(2L<\beta\) is sufficient for Assumption
\ref{ass:w1contr}.

Write \(M=\|\psi\|_\infty\). For \(R\ge M/\beta\), the
closed strip
\[
 D_R:=\{(x,b):|x-b|\le R\}
\]
is invariant, since
\[
 |\Phi(x,y;b,c)-b|
 \le(1-\beta)|x-b|+M\le R.
\]
For \(q>1\), on this domain,
\[
 |\Phi(x,y;b,c)|^q\le2^{q-1}(|b|^q+R^q).
\]
Thus a finite \(q\)-th anchor moment gives Assumption
\ref{ass:qdrift} with \(a_x=a_y=c_c=0\),
\(c_b=2^{q-1}\), and \(c_0=2^{q-1}R^q\).
Theorem \ref{thm:main-w1} applies on \(D_R\) whenever
\(2L<\beta\) and \(f_0\in\mathcal P_{\rho,q}(D_R)\).

The Lipschitz estimate above also gives contraction in mean square.
With \(\ell=1-\beta+2L<1\), weighted Cauchy--Schwarz yields
Assumption \ref{ass:square} with
\[
 L_x^2=\ell(1-\beta+L),\qquad L_y^2=\ell L,
 \qquad\kappa^2=\ell^2<1.
\]
Corollary \ref{cor:w2} therefore applies if the anchors have a
finite moment of some order \(r>2\); on \(D_R\), this also
bounds the initial opinion moment. If \(\rho([-K,K])=1\) for
some \(K<\infty\), one can instead work on the bounded
invariant domain \(D_R\cap(\R\times[-K,K])\).

\subsection{A critical pair-consensus rule}
\label{subsec:critical-consensus}

In this example, both agents adopt one of the two opinions present
at an interaction, chosen with equal probability. This is a version
of neutral copying dynamics \citep{CliffordSudbury1973,HolleyLiggett1975}.
Take \(D=[0,1]\times\{0\}\), \(\rho=\delta_0\), and a
mark \(U\) with \(\Pp(U=-1)=\Pp(U=+1)=1/2\). With
\(\iota=\mathrm{Id}\), set
\begin{equation}\label{eq:critical-consensus-rule}
 \Phi(x,y;0,0,U)=
 \begin{cases}
  x\wedge y,&U=-1,\\
  x\vee y,&U=+1.
 \end{cases}
\end{equation}
\begin{proposition}[A critical counterexample]
\label{prop:critical-consensus-sharpness}
The rule \eqref{eq:critical-consensus-rule} satisfies Assumptions
\ref{ass:inv} and \ref{ass:qdrift} for every \(q>1\).
It satisfies the non-strict version of Assumption
\ref{ass:w1contr} with \(\ell_x=\ell_y=1/2\), but cannot
satisfy \eqref{eq:w1contr} with \(\ell<1\).

Let \(m_0\in\mathcal P([0,1])\) be a non-Dirac opinion law
and put \(f_0=m_0\otimes\delta_0\). The nonlinear solution
is stationary: \(f_t=f_0\) for all \(t\ge0\).
Start the finite system from \(f_0^{\otimes N}\). It reaches
consensus almost surely, with final opinion of law \(m_0\), and
\begin{equation}\label{eq:critical-consensus-failure}
\begin{aligned}
 &\liminf_{N\to\infty}\sup_{t\ge0}
 W_{1,2}\bigl(\Law(Z_1^N(t),Z_2^N(t)),f_t^{\otimes2}\bigr)\\
 &\qquad\ge\iint_{[0,1]^2}|x-y|\,m_0(dx)m_0(dy)>0.
\end{aligned}
\end{equation}
Consequently, the main theorem is false in general if
\(\ell<1\) is replaced by \(\ell\le1\).
\end{proposition}

\begin{proof}
The update stays in \([0,1]\). Since \(|\Phi|^q\le1\),
the drift condition holds with \(c_0=1\) and all other
coefficients zero. The sorting inequality
\[
 |(x\wedge y)-(x'\wedge y')|
 +|(x\vee y)-(x'\vee y')|
 \le|x-x'|+|y-y'|
\]
gives
\[
 \E|\Phi(x,y;0,0,U)-\Phi(x',y';0,0,U)|
 \le\tfrac12|x-x'|+\tfrac12|y-y'|.
\]
Taking \((x,y)=(0,0)\) and \((x',y')=(1,1)\) shows that
any such bound must have \(\ell_x+\ell_y\ge1\).

Let \(X,Y\) be independent with a common law \(m\) on
\([0,1]\). For every bounded measurable \(\varphi\),
\[
\begin{aligned}
 \E\varphi(\Phi(X,Y;0,0,U))
 &=\tfrac12\E[\varphi(X\wedge Y)+\varphi(X\vee Y)]\\
 &=\tfrac12\E[\varphi(X)+\varphi(Y)]
 =\int\varphi\,dm.
\end{aligned}
\]
Hence \(T(m\otimes\delta_0)=m\otimes\delta_0\), proving
the claim about the nonlinear law.

Condition on the finite initial population. Updates create no new
opinion values, so only finitely many configurations are possible.
From any configuration, a sequence of pair updates can make every
agent adopt one opinion already present, with positive probability.
Thus no class of non-consensus states is closed, and consensus is
reached almost surely. Denote the final opinion by \(V_N\).

For bounded measurable \(\varphi:[0,1]\to\R\), the average
\(N^{-1}\sum_i\varphi(X_i^N(t))\) is a bounded martingale:
by the same identity used above, the expected sum over an
interacting pair is unchanged. Letting \(t\to\infty\) gives
\[
 \E[\varphi(V_N)\mid X_1^N(0),\ldots,X_N^N(0)]
 =\frac1N\sum_i\varphi(X_i^N(0)).
\]
Averaging over the independent initial opinions shows that
\(\Law(V_N)=m_0\).

The function \(h(z_1,z_2)=|x_1-x_2|\) is 1-Lipschitz for
\(d_{1,2}\). Its expectation in the finite system tends to zero
by consensus, whereas
\[
 \int h\,df_0^{\otimes2}
 =\iint|x-y|\,m_0(dx)m_0(dy)>0.
\]
The resulting lower bound on \(W_{1,2}\) holds after taking
the supremum in time for every \(N\ge2\). Taking the lower
limit in \(N\) proves \eqref{eq:critical-consensus-failure}.
\end{proof}

\subsection{Pure DeGroot gossip}\label{subsec:degroot-gossip}

Uniform-in-time approximation can still hold when \(\ell=1\).
For the averaging rule
\citep{DeGroot1974,Boyd2006}
\[
 \Phi(x,y)=(1-\omega)x+\omega y,\qquad0<\omega<1,
\]
the conserved empirical mean provides the needed control.
Fix zero anchors and start from \(f_0^{\otimes N}\), where
\(f_0=m_0\otimes\delta_0\) and \(m_0\) has mean \(m\)
and variance \(\sigma_0^2<\infty\). Write
\(f_t=m_t\otimes\delta_0\) for the nonlinear law.

Convexity bounds the second moments of both systems by
\(m^2+\sigma_0^2\). The nonlinear mean stays equal to \(m\),
and the rate-2 equation gives
\[
 \operatorname{Var}(m_t)
 =\sigma_0^2e^{-4\omega(1-\omega)t}.
\]
The finite system preserves
\(\bar X_N=N^{-1}\sum_iX_i^N(0)\). Put
\[
 V_N(t):=\frac1N\sum_i|X_i^N(t)-\bar X_N|^2.
\]
A pair update changes \(V_N\) by
\(-2\omega(1-\omega)(x_i-x_j)^2/N\).
Summing over the pair rates and using
\(\sum_{i<j}(x_i-x_j)^2=N^2V_N\) gives
\[
 \E V_N(t)
 =\frac{N-1}{N}\sigma_0^2
      e^{-4\omega(1-\omega)Nt/(N-1)}
 \le\sigma_0^2e^{-4\omega(1-\omega)t}.
\]
For fixed \(k\ge1\) and \(N\ge\max\{2,k\}\), exchangeability
and \(\E|\bar X_N-m|\le\sigma_0/\sqrt N\) give, by comparison
of both laws with \(\delta_{(m,0)}^{\otimes k}\),
\[
\begin{aligned}
 &W_{1,k}\bigl(\Law(Z_1^N(t),\ldots,Z_k^N(t)),f_t^{\otimes k}\bigr)\\
 &\qquad\le2k\sigma_0e^{-2\omega(1-\omega)t}
             +\frac{k\sigma_0}{\sqrt N}.
\end{aligned}
\]
On each fixed interval, Theorem \ref{thm:finite-time-main} and
Lemma \ref{lem:interp}, with \(p=1,r=2\), give convergence
in \(W_{1,k}\) from the uniform second-moment bound.
The same split at time \(T\) as in Section
\ref{subsec:proof-main}, followed by \(N\to\infty\) and
then \(T\to\infty\), proves
\[
 \lim_{N\to\infty}\sup_{t\ge0}
 W_{1,k}\bigl(\Law(Z_1^N(t),\ldots,Z_k^N(t)),f_t^{\otimes k}\bigr)=0.
\]

\subsection{A hard confidence threshold}\label{subsec:hard-threshold}

Consider the anchored variant of the Deffuant--Weisbuch rule
\citep{Deffuant2000},
\[
 \Phi(x,y;b,c)
 =x+\beta(b-x)+\mu\1_{\{|x-y|\le r\}}(y-x),
\]
with \(0\le\beta\le1\), \(0\le\mu\le1-\beta\), and
\(r\ge0\). These ranges preserve any closed opinion interval
that contains the anchors. The rule is measurable, so Theorem
\ref{thm:finite-time-main} applies with the threshold exactly
as written.

The obstruction to our contraction assumptions can be seen directly.
On \(E\), if \(\mu,r>0\), then for every \(\eps>0\),
\[
 |\Phi(0,r;0,0)-\Phi(0,r+\eps;0,0)|=\mu r.
\]
The input difference tends to zero, so neither Assumption
\ref{ass:w1contr} nor Assumption \ref{ass:square} can hold
on this domain. Thus our results do not give uniform-in-time
approximation across such a confidence boundary.
When the threshold covers the whole opinion interval, the update
reduces to the linear rule in Appendix \ref{subsec:linear-rules}.
The partner term also vanishes if \(\mu=0\) or \(r=0\).

\section{Technical proofs}\label{app:technical}

We give the supporting proofs used in Section \ref{sec:proofs},
along with an extension of the finite-time estimate to non-product
initial laws. The finite-system generator is recorded at the end
for reference.

\subsection{Elementary total-variation facts}
\label{subsec:technical-tv}

\begin{lemma}[Elementary total-variation facts]\label{lem:tvfacts}
Let \((S,\mathcal S)\) and \((S',\mathcal S')\) be measurable spaces.
\begin{longlist}[(iii)]
\item[(i)] If \(\mu,\nu,\alpha,\beta\in\mathcal P(S)\), then
\[
 \|\mu\otimes\alpha-\nu\otimes\beta\|_{\TV}
 \le \|\mu-\nu\|_{\TV}+\|\alpha-\beta\|_{\TV}.
\]
In particular,
\(\|\mu\otimes\mu-\nu\otimes\nu\|_{\TV}
\le2\|\mu-\nu\|_{\TV}\).
\item[(ii)] If \(X,Y\) are \(S\)-valued random variables on the same
probability space, with laws \(\mu,\nu\), and \(\{X\ne Y\}\) is
measurable, then
\[
 \|\mu-\nu\|_{\TV}\le2\Pp(X\ne Y).
\]
\item[(iii)] If \(\mu,\nu\in\mathcal P(S)\) and \(K\) is a Markov
kernel from \(S\) to \(S'\), let \(\mu K\) denote the law
after applying \(K\) to a state with law \(\mu\). Then
\[
 \|\mu K-\nu K\|_{\TV}\le\|\mu-\nu\|_{\TV}.
\]
\end{longlist}
\end{lemma}

\begin{proof}
For \textup{(i)}, use
\[
 \mu\otimes\alpha-\nu\otimes\beta
 =(\mu-\nu)\otimes\alpha+\nu\otimes(\alpha-\beta)
\]
and test against measurable functions bounded in absolute value by
one. For \textup{(ii)}, such a function \(h\) satisfies
\[
 |\E h(X)-\E h(Y)|
 \le\E|h(X)-h(Y)|\le2\Pp(X\ne Y).
\]
For \textup{(iii)}, \(Kh(x):=\int h(y)K(x,dy)\) satisfies
\(\|Kh\|_\infty\le\|h\|_\infty\).
\end{proof}

\subsection{Well-posedness of the nonlinear equation}
\label{subsec:technical-wellposed}

\begin{proof}[Proof of Proposition \ref{prop:wellposed}]
For bounded measurable \(\varphi:E\to\R\), put
\[
 H_\varphi((x,b),(y,c))
 :=\int_U\varphi(\Phi(x,y;b,c,u),b)\,\theta(du).
\]
Since \(\|H_\varphi\|_\infty\le\|\varphi\|_\infty\),
Lemma \ref{lem:tvfacts}\textup{(i)} gives
\begin{equation}\label{eq:T-TV-Lip}
 \|T\mu-T\nu\|_{\TV}
 \le\|\mu\otimes\mu-\nu\otimes\nu\|_{\TV}
 \le2\|\mu-\nu\|_{\TV}.
\end{equation}

Fix \(0<h<1/4\). Let \(\mathcal C_h\) be the space of
total-variation continuous curves from \([0,h]\) to
\(\mathcal P(E)\), with distance
\[
 d_h(g,\widetilde g)
 :=\sup_{0\le t\le h}\|g_t-\widetilde g_t\|_{\TV}.
\]
This space is complete, because probability measures are complete in
total variation and uniform limits preserve continuity. Define
\[
 (\Gamma g)_t
 :=e^{-2t}f_0+2\int_0^t e^{-2(t-s)}Tg_s\,ds.
\]
The integral is a Bochner integral in the space of finite signed
measures with the total-variation norm. By \eqref{eq:T-TV-Lip}, its
integrand is continuous, and \(\Gamma\) maps \(\mathcal C_h\) into
itself. Moreover,
\[
 d_h(\Gamma g,\Gamma\widetilde g)
 \le4h\,d_h(g,\widetilde g).
\]
Banach's theorem gives a unique fixed point. Repeating this
construction on consecutive intervals of length \(h\) gives a unique
global solution of \eqref{eq:mild} that is continuous in total
variation.

Testing \eqref{eq:mild} against \(\varphi\in B_b(E)\) and
differentiating gives \eqref{eq:nl}. Conversely, a weak solution of
\eqref{eq:nl} satisfies
\(\|f_t-f_s\|_{\TV}\le4|t-s|\), since the derivative of
\(\langle\varphi,f_t\rangle\) is bounded by
\(4\|\varphi\|_\infty\). Integrating the equation with the factor
\(e^{2t}\) then gives \eqref{eq:mild}. Thus the weak and mild
formulations have the same unique solution. Taking
\(\varphi(x,b)=\psi(b)\) in \eqref{eq:nl} shows that the anchor
marginal remains equal to its initial value.

For this solution, define the time-dependent transition kernel by
\[
 (K_t\varphi)(x,b)
 :=\int_{E\times U}
 \varphi(\Phi(x,y;b,c,u),b)\,f_t(dy,dc)\,\theta(du).
\]
The kernel is measurable in \((t,x,b)\), by measurability of
\(\Phi\) and total-variation continuity of \(f_t\).
Starting with law \(f_0\), use a rate-\(2\) Poisson clock and apply
\(K_t\) at each ring time \(t\). These kernels determine the law
of the successive states. There are almost surely finitely many rings
on bounded time intervals, so the construction gives a c\`adl\`ag
process, uniquely in law.

Let \(g_t\) be its time-\(t\) law. Conditioning on the last ring
before time \(t\) gives
\[
 g_t=e^{-2t}f_0+2\int_0^t e^{-2(t-s)}g_sK_s\,ds.
\]
The nonlinear solution satisfies the same equation, because
\(f_sK_s=Tf_s\). Lemma \ref{lem:tvfacts}\textup{(iii)} therefore
gives
\[
 \|g_t-f_t\|_{\TV}
 \le2\int_0^t e^{-2(t-s)}\|g_s-f_s\|_{\TV}\,ds.
\]
Gronwall's lemma yields \(g_t=f_t\). This identifies the nonlinear
path law \(Q_{f_0,T}\) and proves the proposition.
\end{proof}

\subsection{An extension to non-product initial laws}
\label{subsec:technical-initial}

The graphical construction also gives a finite-time estimate when
small sets of initial coordinates are approximately independent.

\begin{proposition}[Finite-time estimate for non-product initial laws]
\label{prop:localized-initial}
Fix \(N\ge2\), \(0\le T<\infty\), an initial law
\(\Pi_0^N\in\mathcal P(E^N)\), and \(f_0\in\mathcal P(E)\).
For \(1\le m\le N\), set
\[
 \Delta_{N,m}(f_0;\Pi_0^N)
 :=\sup_{I\subset\{1,\ldots,N\},\ 1\le|I|\le m}
 \left\|\Pi_0^{N,I}-f_0^{\otimes|I|}\right\|_{\TV},
\]
where \(\Pi_0^{N,I}\) is the marginal on the coordinates in
\(I\), listed in increasing order. For \(1\le k\le N\) and
distinct \(i_1,\ldots,i_k\),
\[
\begin{aligned}
 &\left\|\Law_{\Pi_0^N}\bigl(
 (Z_{i_1}^N(s))_{0\le s\le T},\ldots,
 (Z_{i_k}^N(s))_{0\le s\le T}\bigr)
 -Q_{f_0,T}^{\otimes k}\right\|_{\TV}\\
 &\qquad\le\Delta_{N,m}(f_0;\Pi_0^N)
 +\frac{2ke^{2T}}{m}+\eps_{N,k}(T).
\end{aligned}
\]
\end{proposition}

\begin{proof}
Use the same law of the pair clocks and marks for the systems started
from \(\Pi_0^N\) and \(f_0^{\otimes N}\). These driving variables
are independent of the initial states. Let \(V_T\) be the union of
the initial labels required by the backward cones of
\(i_1,\ldots,i_k\), and put \(L_T=|V_T|\).

Fix a realization of the driving variables. By Lemma
\ref{lem:cone-reconstruction}, the tagged paths are a measurable
function of the initial coordinates in \(V_T\). If \(L_T\le m\),
the two conditional path laws therefore differ in total variation by
at most \(\Delta_{N,m}(f_0;\Pi_0^N)\). On the complementary event,
the difference is at most two. Integrating over the common driving
variables shows that the unconditional path laws differ by at most
\[
 \Delta_{N,m}(f_0;\Pi_0^N)+2\Pp(L_T>m).
\]

When the union contains \(l\) labels, new labels are added at rate
\(2l(N-l)/(N-1)\le2l\). Since it starts with \(k\) labels,
\(\E L_T\le ke^{2T}\). Markov's inequality yields
\(\Pp(L_T>m)\le ke^{2T}/m\). Finally, Theorem
\ref{thm:finite-time-main} compares the system started from
\(f_0^{\otimes N}\) with
\(Q_{f_0,T}^{\otimes k}\). The triangle inequality gives the stated
bound.
\end{proof}

\subsection{Uniform moment bounds}
\label{subsec:technical-moments}

\begin{proof}[Proof of Lemma \ref{lem:moment}]
Let \(a,C,K\) be as in the statement, and write
\[
 M_0=\int|x|^q f_0(dx,db),\qquad
 \bar b_q=\int|b|^q\rho(db).
\]
For every \(g\in\mathcal P_{\rho,q}(D)\), the drift assumption gives
\begin{equation}\label{eq:T-moment-drift}
 \int|x|^q\,Tg(dx,db)
 \le a\int|x|^q\,g(dx,db)+C.
\end{equation}

We first obtain the nonlinear bound from the Picard construction in
the proof of Proposition \ref{prop:wellposed}. Fix \(0<h<1/4\), and
suppose that \(\int|x|^qf_{t_0}(dx,db)\le K\). On \([0,h]\), set
\[
\begin{aligned}
 g_s^{(0)}&=f_{t_0},\\
 g_s^{(n+1)}
 &=e^{-2s}f_{t_0}
 +2\int_0^s e^{-2(s-r)}Tg_r^{(n)}\,dr.
\end{aligned}
\]
The iterates remain supported on \(D\) and retain anchor marginal
\(\rho\). Since \(aK+C\le K\), equation
\eqref{eq:T-moment-drift} gives, by induction,
\[
 \int|x|^qg_s^{(n+1)}(dx,db)
 \le e^{-2s}K+(1-e^{-2s})(aK+C)\le K,
 \qquad 0\le s\le h.
\]
The iterates converge in total variation to \(f_{t_0+s}\).
Passing to the limit against \(|x|^q\wedge R\), and then letting
\(R\to\infty\), preserves the moment bound. Starting at
\(t_0=0\) and repeating the argument on successive intervals gives
\begin{equation}\label{eq:nonlinear-uniform-moment}
 \sup_{t\ge0}\int|x|^qf_t(dx,db)\le K.
\end{equation}

For the finite system, let \(\widehat X^N(m)\) be the opinions
after \(m\) pair updates. Let \(\mathcal F_m\) be generated by the
initial state and the first \(m\) updates, and let \(P_{m+1}\) be the pair
chosen at the next update. Write
\[
 S_m=\frac1N\sum_{i=1}^N|\widehat X_i^N(m)|^q,
 \qquad
 \overline B_{q,N}=\frac1N\sum_{i=1}^N|B_i|^q.
\]
Starting from the integrable \(S_0\), the drift bound makes
\(S_{m+1}\) integrable whenever \(S_m\) is. We may therefore apply
\eqref{eq:qdrift} to both agents in the selected pair and subtract
their current moments. For the second agent, the same estimate
follows from \(\theta\circ\iota^{-1}=\theta\). Thus
\[
\begin{aligned}
 &N\E[S_{m+1}-S_m\mid\mathcal F_m,P_{m+1}=\{i,j\}]\\
 &\quad\le
 -(1-a)\bigl(|\widehat X_i^N(m)|^q+|\widehat X_j^N(m)|^q\bigr)\\
 &\qquad +(c_b+c_c)(|B_i|^q+|B_j|^q)+2c_0.
\end{aligned}
\]
Averaging over the uniformly selected unordered pair gives
\begin{equation}\label{eq:embedded-moment-recursion}
 \E[S_{m+1}\mid\mathcal F_m]
 \le\left(1-\frac{2(1-a)}N\right)S_m
 +\frac2N\bigl((c_b+c_c)\overline B_{q,N}+c_0\bigr).
\end{equation}
Since \(N\ge2\) and \(0\le a<1\),
\[
 r_N:=1-\frac{2(1-a)}N\in[0,1).
\]
Iterating \eqref{eq:embedded-moment-recursion}, with \(r_N^0=1\), gives
\begin{equation}\label{eq:embedded-moment-iteration}
 \E[S_m\mid Z^N(0)]
 \le r_N^mS_0
 +(1-r_N^m)\frac{(c_b+c_c)\overline B_{q,N}+c_0}{1-a}.
\end{equation}
The superposition of the pair clocks has rate \(N\). Hence the
number \(J_t\) of pair updates by time \(t\) is Poisson with mean
\(Nt\), independently of the embedded chain and its initial state.
In particular,
\[
 \E[r_N^{J_t}]=e^{Nt(r_N-1)}=e^{-2(1-a)t}.
\]
Since \(\E S_0=M_0\) and
\(\E\overline B_{q,N}=\bar b_q\), equation
\eqref{eq:embedded-moment-iteration} yields
\begin{equation}\label{eq:finite-uniform-moment}
 \frac1N\sum_{i=1}^N\E|X_i^N(t)|^q
 \le e^{-2(1-a)t}M_0
 +(1-e^{-2(1-a)t})\frac{C}{1-a}
 \le K.
\end{equation}

The product initial law is exchangeable. Permuting the labels also
preserves the transition rule: pair selection is uniform, and a
reversal of the order within a pair is absorbed by \(u\mapsto\iota u\).
Thus the law remains exchangeable, and each individual moment equals
the average in \eqref{eq:finite-uniform-moment}.
\end{proof}

\subsection{Interpolation from total variation to Wasserstein}
\label{subsec:technical-interpolation}

\begin{proof}[Proof of Lemma \ref{lem:interp}]
Choose a maximal coupling \((Z,\widetilde Z)\) of \(\mu\) and
\(\nu\). Under our total-variation convention,
\[
 \delta:=\Pp(Z\ne\widetilde Z)
 =\frac12\|\mu-\nu\|_{\TV}.
\]
By H\"older's inequality,
\[
\begin{aligned}
 W_{p,k}(\mu,\nu)^p
 &\le\E\bigl[d_{p,k}(Z,\widetilde Z)^p
                  \1_{\{Z\ne\widetilde Z\}}\bigr]\\
 &\le\bigl(\E d_{p,k}(Z,\widetilde Z)^r\bigr)^{p/r}
        \delta^{1-p/r}.
\end{aligned}
\]
The triangle inequality gives
\[
 \E d_{p,k}(Z,\widetilde Z)^r
 \le2^{r-1}\E\bigl[d_{p,k}(Z,0)^r+d_{p,k}(\widetilde Z,0)^r\bigr]
 \le2^{r-1}M.
\]
Taking \(p\)-th roots proves the first estimate; the constant
\(2^{1-1/p}\) suffices. If the common support has diameter at most
\(R\), then
\[
 W_{p,k}(\mu,\nu)^p\le R^p\delta,
\]
so
\[
 W_{p,k}(\mu,\nu)
 \le R\left(\frac12\|\mu-\nu\|_{\TV}\right)^{1/p}.
\]
\end{proof}

\subsection{Completeness with fixed anchors}
\label{subsec:technical-completeness}

\begin{lemma}[Completeness with fixed anchors]
\label{lem:conditional-completeness}
Let \(D\subseteq E\) be closed, \(p\ge1\), and
\(\rho\in\mathcal P(\R)\). Then
\(\mathcal P_{\rho,p}(D)\) is complete under \(\W_p^\rho\).
For each fixed \(N\ge2\), the laws on \(D^N\) with anchor
marginal \(\rho^{\otimes N}\) and finite moment
\[
 \int\frac1N\sum_{i=1}^N|x_i|^p\,\Pi(dz)<\infty
\]
also form a complete metric space under
\[
 \mathcal W_{p,N}(\Pi,\widetilde\Pi)
 :=\left(\inf\E\left[
 \frac1N\sum_{i=1}^N|X_i-\widetilde X_i|^p
 \right]\right)^{1/p}.
\]
Here the infimum is over couplings of \(\Pi\) and
\(\widetilde\Pi\) whose anchors agree coordinatewise.
Convergence in either metric implies weak convergence on the
corresponding full opinion--anchor space. No moment assumption on
\(\rho\) is needed.
\end{lemma}

\begin{proof}
Disintegration over the common anchor law gives admissible couplings,
whose costs are finite by the opinion moment assumption. Gluing
couplings and applying Minkowski's inequality gives the triangle
inequality for both distances.

Let \((\mu_n)\) be Cauchy under \(\W_p^\rho\). Pass to a
subsequence, still denoted by \((\mu_n)\), with
\(\W_p^\rho(\mu_{n+1},\mu_n)\le2^{-n}\). For each consecutive
pair, choose an anchor-preserving coupling satisfying
\[
 \|X_{n+1}-X_n\|_{L^p}
 \le\W_p^\rho(\mu_{n+1},\mu_n)+2^{-n}
 \le2^{1-n}.
\]
The gluing lemma realizes these couplings on one probability space
with a common anchor \(B\), so that \((X_n,B)\sim\mu_n\) for
every \(n\). Since
\[
 \sum_{n\ge1}\|X_{n+1}-X_n\|_{L^p}<\infty,
\]
the sequence \((X_n)\) converges in \(L^p\) to some \(X\).
Along a further subsequence, it converges almost surely. Closedness of
\(D\) then gives \((X,B)\in D\) almost surely. Thus
\(\mu:=\Law(X,B)\) belongs to \(\mathcal P_{\rho,p}(D)\), and
the common-anchor coupling gives
\[
 \W_p^\rho(\mu_n,\mu)\le\|X_n-X\|_{L^p}\longrightarrow0.
\]
The original Cauchy sequence converges to the same limit by the
triangle inequality.

For the finite-dimensional statement, repeat the argument with the
opinion norm
\[
 \|x\|_{p,N}:=\left(\frac1N\sum_{i=1}^N|x_i|^p\right)^{1/p}
\]
and glue over the common anchor vector. Completeness of
\(L^p\) for this finite-dimensional norm and closedness of \(D^N\)
give the required limit.

Finally, for every bounded Lipschitz \(\varphi:E\to\R\), a
common-anchor coupling yields
\[
 |\langle\varphi,\mu\rangle-\langle\varphi,\nu\rangle|
 \le\Lip(\varphi)\,\W_p^\rho(\mu,\nu).
\]
For a bounded function on \(E^N\) that is Lipschitz under
\(d_{p,N}\), the analogous bound has right-hand side
\(N^{1/p}\Lip(\varphi)\mathcal W_{p,N}(\Pi,\widetilde\Pi)\).
These bounds show that zero distance forces equality of the laws,
and that convergence in either metric implies weak convergence.
\end{proof}

\subsection{Finite-system generator}\label{subsec:finite-generator}

For reference, we record the generator on bounded measurable
functions. For \(z\in E^N\), write \(z_i=(x_i,b_i)\).
Given \(1\le i<j\le N\) and \(u\in U\), let \(z^{i,j,u}\)
be obtained by replacing coordinates \(i\) and \(j\) with
\[
 (\Phi(x_i,x_j;b_i,b_j,u),b_i),
 \qquad
 (\Phi(x_j,x_i;b_j,b_i,\iota u),b_j),
\]
respectively. For \(\varphi\in B_b(E^N)\),
\begin{equation}\label{eq:finite-generator}
 (L^N\varphi)(z)
 =\frac{2}{N-1}\sum_{1\le i<j\le N}\int_U
 \left[\varphi(z^{i,j,u})-\varphi(z)\right]\theta(du).
\end{equation}

\end{appendix}

\begin{funding}
Jingyi Zhang was supported in part by the Georgia Tech ARC-ACO Fellowship.
\end{funding}

\begin{supplement}
\stitle{Code for the numerical illustrations}
\sdescription{Python code for the numerical illustrations in Figures 1--3,
including the simulation parameters, random seeds, smoothing and plotting
routines.
}
\end{supplement}

\end{document}